\documentclass{amsart}

\usepackage{graphicx}
\usepackage{amsfonts}
\usepackage{amssymb}
\usepackage{amscd}
\usepackage{enumerate}
\usepackage{mathrsfs}
\usepackage{braket}

\makeatletter
\renewcommand{\section}{%
\@startsection{section}{1}%
  \z@{.7\linespacing\@plus\linespacing}{.5\linespacing}%
 {\normalfont\large\bfseries\centering}}
\makeatother

\newtheorem{theorem}{Theorem}[section]

\newtheorem{lemma}[theorem]{Lemma}
\newtheorem{proposition}[theorem]{Proposition}

\newtheorem*{theorem*}{Theorem} 
\newtheorem*{corollary*}{Corollary}
\newtheorem*{conjecture*}{Conjecture}
\newtheorem*{lemma*}{Lemma}
\newtheorem*{proposition*}{Proposition}
\newtheorem*{problem*}{Problem}
\newtheorem*{axiom*}{Axiom}
\newtheorem*{example*}{Example}
\newtheorem*{exercise*}{Exercise}
\newtheorem*{definition*}{Definition}

\theoremstyle{definition}
\newtheorem{remark}{Remark}[section]
\newtheorem*{remark*}{Remark}
\numberwithin{equation}{section} 

\renewcommand{\l}{\left}
\renewcommand{\r}{\right}
\newcommand{\eps}{\varepsilon}

\newcommand{\R}{{\mathbb R}}

\newcommand{\Z}{{\mathbb Z}}
\newcommand{\T}{{\mathbb T}}
\newcommand{\im}{{\rm Im}}
\newcommand{\re}{{\rm Re}}

\newcommand{\ds}{\displaystyle}

\newcommand{\del}{\partial}

\def\norm[#1]{\left\Vert #1 \right\Vert}
\def\tbra[#1,#2]{\left\langle #1 , #2\right\rangle} 
\def\rbra[#1,#2]{\left( #1 , #2 \right)} 
\def\sbra[#1,#2]{\left[ #1 , #2 \right]} 

\newcommand{\dn}{\mathrm{dn}}
\newcommand{\sn}{\mathrm{sn}}

\newcommand{\scG}{{\mathscr G}}

\newcommand{\cE}{{\mathcal E}}

\newcommand{\cJ}{{\mathcal J}}

\newcommand{\cM}{{\mathcal M}}

\newcommand{\cP}{{\mathcal P}}

\newcommand{\cS}{{\mathcal S}}

\begin{document}

\title[Long-period limit of periodic traveling wave solutions]{Long-period limit of exact periodic traveling wave solutions for the derivative nonlinear Schr\"{o}dinger equation}


\author{Masayuki Hayashi}
\address{Department of Applied Physics, Waseda University, Tokyo 169-8555, Japan}
\curraddr{}
\email{masayuki-884@fuji.waseda.jp}
\thanks{}

\subjclass[2010]{Primary 35Q55, 35C07; Secondary 33E05}
\keywords{derivative nonlinear Schr\"{o}dinger equation, solitons, periodic traveling waves, long-period limit, elliptic functions and elliptic integrals}

\date{}

\dedicatory{}


\begin{abstract}
We study the periodic traveling wave solutions of the derivative nonlinear Schr\"{o}dinger equation (DNLS). It is known that DNLS has two types of solitons on the whole line; one has exponential decay and the other has algebraic decay. The latter corresponds to the soliton for the massless case. In the new global results recently obtained by Fukaya, Hayashi and Inui \cite{FHI17}, the properties of two-parameter of the solitons are essentially used in the proof, and especially the soliton for the massless case plays an important role. To investigate further properties of the solitons, we construct exact periodic traveling wave solutions which yield the solitons on the whole line including the massless case in the long-period limit. Moreover, we study the regularity of the convergence of these exact solutions in the long-period limit. Throughout the paper, the theory of elliptic functions and elliptic integrals is used in the calculation.
\end{abstract}

\maketitle

\tableofcontents
\section{Introduction}
\subsection{Background}
We begin with the following equation
\begin{equation}
\label{eq:1.1}
i \partial_t \psi + \partial_x^2 \psi +i  \partial_x( |\psi |^{2} \psi )=0 , 
\end{equation}
which is known as a derivative nonlinear Schr\"{o}dinger equation. This equation appears in plasma physics as a model for the propagation of Alfv\'{e}n waves in magnetized plasma (see \cite{MOMT76, M76}) and it is known to be completely integrable (see \cite{KN78}). 

There is a large literature on the Cauchy problem for the equation (\ref{eq:1.1}). Tsutsumi and Fukuda \cite{TF80, TF81} studied the well-posedness in $H^s(\R)$ for $s>3/2$ by classical energy method which depends on parabolic regularization. The well-posedness in the energy space $H^1(\R)$ was first proved by Hayashi \cite{H93}. He introduced gauge transformation (see e.g. (\ref{eq:1.2}) or (\ref{eq:1.14}) below) to overcome the derivative loss. In a later work, Hayashi and Ozawa \cite{HO92} proved the solution of $H^1(\R)$ is global if the initial data $\psi_0$ satisfies $\| \psi_0\|_{L^2}^2 < 2\pi$. Recently, Wu \cite{Wu15} improved this global result, more specifically, he proved the solution is global if the initial data satisfies $\|\psi_0\|_{L^2}^2 < 4\pi$. We will discuss connection between these global results and solitons of (\ref{eq:1.1}) later. For the Cauchy problem for (\ref{eq:1.1}) in $H^s(\R )$ with $s<1$, we refer to \cite{T99, BL01, CKSTT01, CKSTT02, GW17}.

There are several forms of (\ref{eq:1.1}) that are equivalent under a gauge transformation. By using the following gauge transformation to the solution of (\ref{eq:1.1})
\begin{align}
\label{eq:1.2}
v(t,x) = \psi (t,x)\exp\l( \frac{i}{2} \int_{-\infty}^{x} |\psi (t,x)|^2 dx\r) ,
\end{align}
then $v$ satisfies the following equation:
\begin{align}
\label{DNLS}
 i \del_t v + \del_x^2 v +i |v|^{2} \del_x v =0,
\end{align}
This equation has the following conserved quantities:  
\begin{align*}
\tag{Energy}
 E(v)&:=\frac{1}{2}\| \partial_x v  \|_{L^2}^2 - \frac{1}{4} \re \int i |v|^{2}\overline{v}\partial_x v dx,
\\
\tag{Mass}
 M(v)&:= \| v\|_{L^2}^2,
\\
\tag{Momentum}
P(v)&:=\re \int i\partial_x v \overline{v} dx.
\end{align*} 
The equation (\ref{DNLS}) can be rewritten as 
\begin{align}
\label{eq:1.4}
i \partial_t v = E ' (v).
\end{align}
The Hamiltonian form (\ref{eq:1.4}) is useful when one considers problems of orbital stability/instability of solitons. It is known that (\ref{DNLS}) has a two-parameter family of solitons (see \cite{KN78, CO06, G08, L09})
\begin{align}
\label{eq:1.5}
v_{\omega,c}(t,x)=e^{i\omega t} \phi_{\omega,c}(x-ct),
\end{align}
where $(\omega,c)$ satisfies $\omega > c^2/4$, or $\omega = c^2/4$ and $c>0$, and
\begin{align} 
\label{eq:1.6}
\phi_{\omega,c}(x) 
&= \Phi_{\omega,c}(x) \exp\l( i\frac{c}{2}x - \frac{i}{4} \int_{-\infty}^{x} \Phi_{\omega,c}(y)^{2} dy \r),
\\
\label{eq:1.7}
\Phi_{\omega,c}^2(x) &=
\l\{
\begin{array}{ll}\ds
 \frac{4\omega - c^2}{\sqrt{\omega} \l( \cosh (\sqrt{4\omega- c^2}x)-\frac{c}{2\sqrt{\omega}} \r) } 
& \ds \text{if } \omega > \frac{c^2}{4},
\\
 & 
\\ \ds
\frac{4 c}{(cx)^2+1} & \ds \text{if } \omega =\frac{c^2}{4} \text{ and } c>0.
\end{array}
\r.
\end{align}
We note that $\Phi_{\omega,c}$ is the positive radial (even) solution of 
\begin{align}
\label{eq:1.8}
- \Phi ''+ \l(\omega- \frac{c^2}{4}\r) \Phi +\frac{c}{2} |\Phi|^{2} \Phi - \frac{3}{16} |\Phi|^{4}\Phi =0,
\end{align}
and the complex-valued function $\phi_{\omega,c}$ is the solution of 
\begin{align}
\label{eq:1.9}
-\phi ''+ \omega \phi +ic \phi ' - i|\phi|^{2} \phi '=0. 
\end{align}
The equation (\ref{eq:1.9}) can be rewritten as $S_{\omega ,c}' (\phi) =0$, where
\begin{align*}
S_{\omega ,c}(\phi ):=E(\phi )+\frac{\omega}{2}M(\phi )+\frac{c}{2} P(\phi ).
\end{align*}
The condition of two parameters $(\omega,c)$
\begin{equation} 
\label{WC}
\omega > c^2/4, \text{ or } \omega = c^2/4 \text{ and } c>0
\end{equation}
is a necessary and sufficient condition for the existence of non-trivial solutions of (\ref{eq:1.8}) vanishing at infinity (see Appendix A in \cite{FHI17} or \cite{BeL83}). 
Guo and Wu \cite{GW95} proved that the soliton $u_{\omega,c}$ is orbitally stable when $\omega>c^2/4$ and $c<0$ by applying the abstract theory of Grillakis, Shatah, and Strauss \cite{GSS87, GSS90}.
Colin and Ohta \cite{CO06} proved that the soliton $u_{\omega,c}$ is orbitally stable when $\omega>c^2/4$ by applying variational characterization to solitons as in Shatah \cite{S83}. The case of $\omega=c^2/4$ and $c>0$ (massless case) is treated\footnote{The ``orbital stability" discussed in \cite{KW18} is different from usual definition. Their result does not contradict that finite time blow-up occurs to the initial data near the soliton for the massless case.} by Kwon and Wu \cite{KW18}, while the orbital stability or instability for the massless case is still an open problem.

From the explicit formulae (\ref{eq:1.6}) and (\ref{eq:1.7}) of solitons, we have
\begin{align}
\label{eq:1.11}
M(\phi_{\omega ,c})=M(\Phi_{\omega ,c}) = 8\tan^{-1} \sqrt{ \frac{2\sqrt{\omega}+c}{2\sqrt{\omega}-c} },
\end{align}
where $(\omega ,c)$ satisfies (\ref{WC}) (see \cite[Lemma 5]{CO06} for the proof). If we consider the curve
\begin{align}
\label{eq:1.12}
c=2s\sqrt{\omega}
\end{align}
for $\omega >0$ and $s\in (-1,1]$, we have
\begin{align*}
\Phi_{\omega ,2s\sqrt{\omega}}(x) =\omega^{\frac{1}{4}} \Phi_{1,2s} (\sqrt{\omega}x).
\end{align*}
This means that the curve (\ref{eq:1.12}) corresponds to the scaling which is invariant of the mass of the soliton. We note that the function
\begin{align}
\label{eq:1.13}
s\mapsto M(\phi_{1 ,2s} ) =8\tan^{-1} \sqrt{ \frac{1+s}{1-s} }
\end{align}
is a strictly increasing function from $(-1,1]$ to $(0,4\pi ]$. Especially, the threshold value $4\pi$
corresponds to the mass of the soliton for the massless case. 

Here, let us review the global results in the energy space $H^1(\R)$. We consider another gauge equivalent form of (\ref{eq:1.1}). By using the following gauge transformation to the solution of (\ref{DNLS})
\begin{align}
\label{eq:1.14}
u(t,x) = v(t,x)\exp\l( \frac{i}{4} \int_{-\infty}^{x} |v(t,x)|^2 dx\r), 
\end{align}
then $u$ satisfies the following equation:
\begin{align}
\label{eq:1.15}
i \partial_t u + \partial_x^2 u +\frac{i}{2} |u|^{2} \del_x u -\frac{i}{2}u^2\del_x \overline{u} +\frac{3}{16}|u|^4u =0 .
\end{align}
Conserved quantities of (\ref{DNLS}) are transformed as follows;
\begin{align}
\label{eq:1.16}
&{\cE}(u)= \frac{1}{2}\| \partial_x u\|_{L^2}^2  -\frac{1}{32} \| u\|_{L^6}^6 ,\\
\label{eq:1.17}
&{\cM}(u) =\| u\|_{L^2}^2, \\  
\label{eq:1.18}
&{\cP}(u)= \re \int i\partial_x u \overline{u} dx +\frac{1}{4}\| u\|_{L^4}^4 .
\end{align}
The gauge transformation (\ref{eq:1.14}) was used in \cite{HO92} to cancel out the interaction term with derivative in the energy functional. Hayashi and Ozawa \cite{HO92} applied the following sharp Gagliardo--Nirenberg inequality
\begin{align}
\label{GN1}
\| f\|_{L^6}^6 \leq \frac{4}{\pi^2}\| f\|_{L^2}^4\| \partial_{x}f\|_{L^2}^2
\end{align}
in order to obtain a priori estimate in $\dot{H}^1(\R )$ by using conservation laws of the mass and the energy. 
They proved the $H^1(\R)$-solution of (\ref{eq:1.15}) is global if the initial data $u_0$ satisfies 
\begin{align}
\label{eq:1.20}
 \| u_0 \|_{L^2}^2 < \| Q \|_{L^2}^2=2\pi ,
\end{align}
where $Q$ is defined by $Q:= \Phi_{1,0}$. We note that $Q$ is an optimal function for the inequality (\ref{GN1}).
This result is closely related to the earlier work by Weinstein \cite{W82} for focusing $L^2$-critical nonlinear Schr\"{o}dinger equations. Consider the following quintic nonlinear Schr\"{o}dinger equation:
\begin{equation}
\label{NLS}
i \partial_t u +\partial_x^2 u + \frac{3}{16}|u|^{4}u=0.
\end{equation}
The equation (\ref{NLS}) has the same energy $\cE (u)$ of (\ref{eq:1.16}) and the same standing wave $e^{it}Q$ as the equation (\ref{eq:1.15}). Furthermore, (\ref{eq:1.15}) and (\ref{NLS}) are $L^2$-critical in the sense that the equation and $L^2$-norm are invariant under the scaling transformation
\begin{align}
\label{eq:1.22}
u_{\gamma}(t,x):=\gamma^{\frac{1}{4}} u(\gamma t,\gamma^{\frac{1}{2}} x), \quad \gamma>0. 
\end{align}
Weinstein \cite{W82} proved that if the initial data of (\ref{NLS}) satisfies the mass condition (\ref{eq:1.20}),
then the $H^1(\R)$-solution is global. In the case (\ref{NLS}), it is known that this mass condition is sharp, in the sense that for any $\rho \geq 2\pi$, there exists $u_0 \in H^1(\R)$ such that $\| u_0\|_{L^2}^2 =\rho$ and such that corresponding solution $u$ to (\ref{NLS}) blows up in finite time. From this analogy, Hayashi and Ozawa \cite{HO92} conjectured that the mass condition (\ref{eq:1.20}) is also sharp for the equation (\ref{eq:1.15}) (equivalently (\ref{eq:1.1}) or (\ref{DNLS})). 

A similar analogy can be seen for the quintic generalized Korteweg-de Vries equation:
\begin{align}
\label{eq:1.23}
\del_t u+\del^3_x u+\frac{3}{16}\del_x (u^5) =0.
\end{align}
This equation is also the $L^2$-critical equation which has the same energy $\cE (u)$ as (\ref{eq:1.15}) and the traveling wave solution $Q(x-t)$. Hence, if the initial data of (\ref{eq:1.23}) satisfies the mass condition (\ref{eq:1.20}), then the $H^1(\R )$-solution is global. It is also known that the solution of (\ref{eq:1.23}) blows up in finite time to the initial data satisfying
\begin{align*}
\cE (u_0) <0,~\cM (Q) <\cM (u_0) <\cM (Q)+\eps
\end{align*} 
for small $\eps >0$ and some decay condition; see \cite{M01, MM02}.

However, the mass condition (\ref{eq:1.20}) is not sharp to the equation (\ref{eq:1.15}). Wu \cite{Wu13, Wu15} took advantage of conservation law of the momentum as well as conservation laws of the mass and the energy. He used the following sharp Gagliardo--Nirenberg inequality
\begin{align}
\label{GN2}
\| f\|_{L^6}^6 \leq 3(2\pi )^{-\frac{2}{3}}\| f\|_{L^4}^{\frac{16}{3}}\| \partial_{x}f\|_{L^2}^{\frac{2}{3}}
\end{align}
in his argument to connect the estimates obtained from the energy (\ref{eq:1.16}) and the momentum (\ref{eq:1.18}) (see also \cite{GW17}).
Then, he proved that the $H^1(\R )$-solution of (\ref{eq:1.15}) is global if the initial data $u_0$ satisfies
\begin{align}
\label{eq:1.25}
\| u_0 \|_{L^2}^2 < \| W\|_{L^2}^2=4\pi ,
\end{align}
where $W$ is defined by $W:= \Phi_{1,2}$. We note that $W$ is an optimal function for the inequality (\ref{GN2}). 

One of the reason why the difference of global results as described above occurs is due to that the equation (\ref{eq:1.15}) has a two-parameter family of solitons. The massless case corresponds to the threshold for the existence of solitons, and the value $4\pi$ corresponds to the mass of the soliton for the massless case. Hence, it is reasonable to conjecture that $4\pi$ is an optimal upper bound of the mass for the global existence of $H^1(\R)$-solutions by the analogy with (\ref{NLS}) and (\ref{eq:1.23}) as $L^2$-critical equations. However, existence of blow-up solutions for the derivative nonlinear Schr\"{o}dinger equation is a large open problem. It is known that finite time blow-up occurs for the equation (\ref{eq:1.1}) on a bounded interval or on the half line, with Dirichlet boundary condition (see \cite{Tan04, Wu13}), but unfortunately one can not apply these proofs to the whole line case. We also refer to \cite{LSS13, CSS16} for numerical approaches to this problem. 

Recently, Fukaya, Hayashi and Inui \cite{FHI17} gave a sufficient condition for global existence by using potential well theory; if the initial data $u_0$ of (\ref{DNLS}) satisfying the following condition that there exists $(\omega ,c)$ satisfying (\ref{WC}) such that
\begin{align}
\label{eq:1.26}
S_{\omega ,c} (u_0) \leq S_{\omega ,c} (\phi_{\omega ,c}) ,~\tbra[ S_{\omega ,c}' (u_0 ), u_0] \geq 0 ,
\end{align}
then the corresponding $H^1(\R)$-solution exists globally in time. For the case $\omega >c^2/4$, this global result is essentially proved in \cite{CO06}. 
In \cite{FHI17}, they mainly proved that there exists $(\omega ,c)$ satisfying both (\ref{WC}) and the condition (\ref{eq:1.26}) 
if the initial data $u_0$ satisfies $\| u_0\|_{L^2}^2 <4\pi$, or $\| u_0\|_{L^2}^2=4\pi$ and $P(u_0)<0$.
This gives a simple alternative proof of Wu's global result.
We note that the latter global result; the global result for the initial data such that
\begin{align*}
\| u_0\|_{L^2}^2=4\pi ~\text{and}~P(u_0) <0
\end{align*}
is first discovered by \cite{FHI17}, and this gives the first progress to investigate the dynamics around the soliton for the massless case. Furthermore, they proved that the condition (\ref{eq:1.26}) contains the initial data in $H^1 (\R )$ with arbitrarily large mass (see Corollary 1.5 in \cite{FHI17}). We note that their proofs are done by essentially using the properties of two-parameter of the solitons, and especially the soliton for the massless case plays an important role in the proof. 

Recently, in \cite{JLPS18} it was proved by inverse scattering approach (see also \cite{LPS16, PSS17, PS17} for related works) that the equation \eqref{DNLS} is globally well-posed for any initial data belonging to weighted Sobolev space $H^{2,2} (\R )$, where
\begin{align*}
H^{2,2} (\R ) := \l\{ u\in H^2 (\R )~;~\braket{\cdot}^2u\in L^2(\R ) \r\}.
\end{align*}
However, the dynamics in the energy space $H^1(\R)$ (especially above the mass threshold $4\pi$) is still unclear. We note that the solitons for the massless case do not contain in $H^{2,2}(\R)$, but they contain in $H^1(\R)$. Therefore, the difference of functional spaces is quite important for (\ref{DNLS}) from the viewpoint of solitons. We also note that the results in \cite{JLPS18} do not imply the nonexistence of blow-up solutions for \eqref{DNLS} in the energy space  $H^1(\R)$; see blow-up criteria in \cite{KW18}.

The equation (\ref{eq:1.1}) in the periodic setting is also an important problem. Tsutsumi and Fukuda \cite{TF80} proved well-posedness in $H^s(\T )$ for $s>3/2$ in the same way as the whole line case, where $\T :=\R/2\pi\Z$. To prove well-posedness in $H^1(\T )$ one can not directly apply the proof in \cite{H93} to the periodic setting since the $L^4$ Strichartz estimate on a torus holds with a loss of $\eps >0$ derivatives (see \cite{Bo93}). Herr \cite{Hr06} proved local well-posedness in $H^s (\T)$ for $s\geq 1/2$ by using periodic gauge transformation and multilinear estimates in Fourier restriction norm spaces (see also \cite{GH08}). In \cite{MO15}, by adapting Wu's proof to the periodic setting they proved that the $H^1(\T)$ solution of (\ref{eq:1.1}) is global if the mass is less than $4\pi$. For global results in $H^s(\T )$ with $s<1$, we refer to \cite{Win10, M17}.  

The periodic traveling waves of the derivative nonlinear Schr\"{o}dinger equation have only been partially studied. 
As a first mathematical work of this problem, Imamura \cite{I10} studied semi-trivial solutions:
\begin{align*}
\phi_{\ell}^c (x-ct) =\sqrt{c-\ell}e^{i\ell (x-ct)}, \quad \ell \in\Z \setminus\{ 0\}, ~c>\ell ,
\end{align*}
which are $2\pi$-periodic traveling wave solutions of (\ref{DNLS}). In \cite{I10}, he proved orbital stability of semi-trivial solutions by applying the abstract theory of Grillakis, Shatah, and Strauss \cite{GSS87, GSS90}.
Murai, Sakamoto and Yotsutani \cite{MSY15} discussed the explicit formulae of periodic traveling waves of (\ref{DNLS}) which are not semi-trivial. If we put the form
\begin{align*}
v(t,x)=e^{i\omega t} U(x-ct)
\end{align*}
into (\ref{DNLS}), then $U$ satisfies the equation
\begin{align}
\label{eq:1.27}
-U'' +\omega U +icU' -i|U|^2U' =0, 
\end{align}
with periodic boundary conditions. By using polar coordinates $U(x)=r(x)e^{i\theta (x)}$, a direct calculation shows that the functions $r(x)$ and $\theta(x)$ satisfy 
\begin{align}
\label{eq:1.28}
&-r ''+\l( \omega -\frac{c^2}{4} +\frac{b}{2}\r) r +\frac{c}{2}r^3 -\frac{3}{16} r^5 +\frac{b^2}{r^3}=0, \\
\label{eq:1.29}
&\theta (x) = \frac{c}{2}x -\frac{1}{4}\int_{0}^{x} r(y)^2 dy+b\int_{0}^{x} \frac{dy}{r(y)^2},
\end{align}
where $b$ is some constant which comes from integration. If we consider solutions vanishing at infinity, we can take $b=0$. In this case (\ref{eq:1.28}) corresponds to the equation (\ref{eq:1.8}), and (\ref{eq:1.29}) corresponds to the gauge transformation (\ref{eq:1.6}). However, in general $b$ is a non-zero constant in the periodic setting. In \cite{MSY15}, they first obtain explicit formulae of all the $2\pi$-periodic solutions of (\ref{eq:1.28}), and then try to find the solutions from among them which satisfy periodic conditions of $\theta$:
\begin{align*}
\theta (0) =0, ~\theta (2\pi ) =2\pi \ell ,
\end{align*}
which is equivalent that
\begin{align}
\label{eq:1.30}
2\pi\ell = c\pi -\frac{1}{4} \int_{0}^{2\pi} r(x)^2 dx +b\int_{0}^{2\pi} \frac{dx}{r(x)^2}, 
\end{align}
where $\ell \in\Z\setminus \{ 0\}$ is a winding number. Since general solutions of (\ref{eq:1.28}) are complicate as can be seen in \cite{MSY15}, it is a quite delicate problem to find the solutions which satisfy the condition (\ref{eq:1.30}). In \cite{MSY15}, partial numerical computations are done to confirm the existence of solutions which satisfy special periodic boundary conditions above.  

The main difficulty to obtain exact periodic traveling wave solutions of (\ref{DNLS}) is that the nonlocal problem as (\ref{eq:1.30}) appears. 
In this paper, to avoid complex calculation in nonlocal issues we consider the equation (\ref{eq:1.15}) on a torus; i.e.,
\begin{align}
\label{DNLS1}
i \partial_t u + \partial_x^2 u +\frac{i}{2} |u|^{2} \partial_x u -\frac{i}{2}u^2\partial_x \overline{u} +\frac{3}{16}|u|^4u =0,
\quad (t,x)\in\R\times\T_{2L},
\end{align}
where $\T_{2L}=\R /2L\Z \simeq [-L ,L]$ is the torus of size $2L$. The energy, mass and momentum of (\ref{DNLS1}) are given by (\ref{eq:1.16}), (\ref{eq:1.17}) and (\ref{eq:1.18}) respectively, in the periodic setting. Our aim of this paper is to find exact periodic traveling wave solutions which yield the solitons on the whole line including the massless case in the long-period limit.
We also study the regularity of the convergence of exact periodic traveling wave solutions in the long-period limit. 
 
In the end of this subsection, we discuss the relation between the equations (\ref{eq:1.1}), (\ref{DNLS}) and (\ref{eq:1.15}) on $\T_{2L}$. Let us recall the periodic gauge transformation introduced by Herr \cite{Hr06}. For $a\in\R$, let $\scG_{a}: L^2(\T_{2L} ) \to L^2(\T_{2L})$ be defined by
\begin{align}
\label{eq:1.32}
\scG_a (f) (x) =e^{ia\cJ (f)(x)} f(x),
\end{align} 
where $\cJ (f)$ 
\begin{align*}
\cJ (f) (x) :=\frac{1}{2L} \int_{0}^{2L}\int_{\theta}^{x} \l( |f(y)|^2 -\mu [f] \r) dyd\theta
\end{align*}
and
\begin{align*}
\mu =\mu [f] :=\frac{1}{2L} \| f\|_{L^2 (\T_{2L})}^2.
\end{align*}
We note that $\cJ (f)$ is the $2L$-periodic primitive of $|f|^2-\mu (f)$ with mean zero. For the solution $u$ of (\ref{DNLS}) on $\T_{2L}$, we define the gauge transformed solution by
\begin{align*}
u(t,x) := \scG_{a}(v) (t, x+2a\mu t).
\end{align*}
A straightforward calculation shows that $u$ satisfies
\begin{align}
\label{eq:1.33}
i\del_t u +\del_x^2 u +(1-2a)i|u|^2\del_x u-2iau^2\del_x\overline{u} +a\l( a+\frac{1}{2}\r) |u|^4u +e_L (u)=0,
\end{align}
where
\begin{align}
\label{eq:1.34}
e_L (u) &:= \psi (u)u -a\mu |u|^2u, \\
\label{eq:1.35}
\psi (u)&:= \frac{a}{2L} \int_{0}^{2L} \l( 2\im (\overline{u}\del_x u) (t, \theta ) +\l( \frac{1}{2} -2a \r) |u|^4 (t ,\theta ) \r) d\theta +a^2\mu^2 .
\end{align}
We note that when $a=-\frac{1}{2}$ [resp. $a=\frac{1}{4}$] the equation (\ref{eq:1.33}) represents (\ref{eq:1.1}) [resp. (\ref{eq:1.15})] on $\T_{2L}$ with some error term $e_L(u)$. Therefore, three equations (\ref{eq:1.1}), (\ref{DNLS}) and (\ref{eq:1.15}) on $\T_{2L}$ can be considered to be 
{\it almost equivalent} under the suitable periodic gauge transformation. Since $e_L$ formally goes to $0$ as $L\to\infty$, it is reasonable to consider that these three equations on $\T_{2L}$ do not have essentially different structure at least when $L$ is sufficiently large. This is compatible with that these three equations on the whole line are gauge equivalent. 
As can be seen in the proof in \cite{Hr06}, 
the error term $e_L(u)$ does not give any difficulty to prove well-posedness. However, it gives a delicate problem when one tries to obtain exact periodic traveling wave solutions, since the error term $e_L(u)$ is nonlocal. Hence, we consider the equation (\ref{DNLS1}) as a basic equation.%


\subsection{Main results}
First, we note that 
\begin{align}
\label{eq:1.36}
u_{\omega,c}(t,x)=e^{i\omega t} \varphi_{\omega,c}(x-ct), 
\end{align}
is a two-parameter family of solitons of the equation (\ref{eq:1.15}) on the whole line, where
\begin{align} 
\label{eq:1.37}
\varphi_{\omega,c}(x) 
&=e^ {i\frac{c}{2}x}\Phi_{\omega,c}(x) ,
\end{align}
and $\Phi_{\omega ,c}$ is defined by (\ref{eq:1.7}).
Note that $\varphi_{\omega ,c}$ satisfies
\begin{align}
\label{eq:1.38}
-\varphi ''+ \omega \varphi +ic \varphi ' +\frac{c}{2}|\varphi |^2\varphi -\frac{3}{16}|\varphi |^4\varphi =0.
\end{align}
We consider the elliptic equation (\ref{eq:1.8}) on a torus: 
\begin{align}
\label{ELL}
- \Phi ''+ \l(\omega- \frac{c^2}{4}\r) \Phi +\frac{c}{2} |\Phi|^{2} \Phi - \frac{3}{16} |\Phi|^{4}\Phi =0,\quad x \in \T_{2L}.
\end{align}
To find exact solutions which yield the solitons in the long-period limit, we need to find positive single-bump solutions of (\ref{ELL}). We have the following theorem.
\begin{theorem}
\label{thm:1.1}
Let $(\omega ,c) \in \R^2$ satisfy \textup{(\ref{WC})}. Assume that $L>0$ satisfies 
\begin{align}
\label{eq:1.40}
L_0 =L_0 (\omega ,c) <L <\infty ,
\end{align}
where $L_0(\omega ,c)$ is a positive constant determined by $(\omega ,c)$ \textup{(}see Remark \textup{\ref{rem:1.2}} below\textup{)}. Then, there exists the positive single-bump solution $\Phi_{\omega ,c}^L$ of \textup{(\ref{ELL})} on $\T_{2L}$ such that  $\Phi_{\omega ,c}^L(x) \to \Phi_{\omega , c}(x)$  for any $x \in \R$ as $L\to \infty$. Furthermore, $\Phi_{\omega,c}^L$ is explicitly represented as
\begin{align}
\label{eq:1.41}
\l( \Phi_{\omega , c}^L (x) \r)^2 = \eta_3 \frac{\mathrm{dn}^2\l( \frac{x}{2g} ; k\r)}{1+\beta^2 \mathrm{sn}^2 \l( \frac{x}{2g} ;k\r)} ~,
\quad x \in [-L,L]
\end{align}
with parameters $\eta_3, g, k , \beta$ depending on $(L, \omega , c)$. 
\end{theorem}
\begin{remark}
\label{rem:1.1}
The value $\eta_3$ corresponds to the maximal value of $\l( \Phi_{\omega ,c}^L \r)^2$. We note that $\eta_3$ satisfies
\begin{align*}
\alpha_0 <\eta_3 <\Phi^2_{\omega ,c} (0),
\end{align*}
where $\alpha_0$ is defined by
\begin{align*}
\alpha_0 := \frac{1}{3} \l( 4c+\sqrt{48\omega +4c^2}\r) .
\end{align*}
It is shown that $\alpha_0$ is a positive constant when $(\omega ,c)$ satisfies (\ref{WC}) (see Lemma \ref{lem:3.1}).
\end{remark}
\begin{remark}
\label{rem:1.2}
$L_0$ is explicitly represented as
\begin{align*}
L_0 =L_0(\omega ,c):=\frac{2\pi}{ \sqrt{\alpha_0 \sqrt{A(\alpha_0)}} },
\end{align*}
where $A(x)$ is defined by
\begin{align*}
A(x) :=-3x^2 +8cx+64\omega .
\end{align*}
We note that $A(\alpha_0 )$ is a positive constant when $(\omega ,c)$ satisfies (\ref{WC}) (see Remark \ref{rem:3.1}). The condition (\ref{eq:1.40}) is optimal in the sense that when $L=L_0$, the constant $\sqrt{\alpha_0}$ is a solution of (\ref{ELL}) and $\Phi_{\omega ,c}^L (x)\to \sqrt{\alpha}$ for any $x\in [-L_0 ,L_0]$ as $L\downarrow L_0$. In short, the condition (\ref{eq:1.40}) is optimal in order that $\Phi_{\omega ,c}^L$ has a single bump.
\end{remark}
The functions $\mathrm{dn}$ (dnoidal) and $\mathrm{cn}$ (cnoidal) in Theorem \ref{thm:1.1} are usual Jacobi's elliptic functions; see Section \ref{sec:2} for a precise definition. We note that if we take $c_L \in \frac{2\pi}{L}\Z$,  
exact periodic traveling waves defined by
\begin{align}
\label{eq:1.42}
u_{\omega, c_L}^L(t,x)=e^{i\omega t+ i\frac{c_L}{2}(x-c_Lt)} \Phi_{\omega,c_L}^L(x-c_L t)
=:e^{i\omega t} \varphi_{\omega ,c_L}^L(x-c_Lt)
\end{align}
satisfy the equation (\ref{DNLS1}) on $\T_{2L}$. If for each $L>L_0$ we take $c_L \in \frac{2\pi}{L}\Z$ such that $c_L\to c$ as $L\to\infty$, we have
\begin{align}
\label{eq:1.43}
\varphi_{\omega ,c_L}^L (x) \to \varphi_{\omega ,c} (x)
\end{align}
for any $x\in\R$ as $L\to\infty$. This gives the pointwise convergence of periodic traveling waves in the long-period limit.

In the one-parameter case ($\omega >0$ and $c=0$), exact solutions defined by 
(\ref{eq:1.41}) correspond to periodic wave solutions to (\ref{NLS}) and (\ref{eq:1.23}) which were studied in \cite{AN09}. Construction of solutions in Theorem \ref{thm:1.1} is done by a simple quadrature method in the similar way as the one-parameter case. However, derivation of the detailed properties of exact solutions in the two-parameter case is far from being obvious from the result of one-parameter case. For instance, we can show that the modulus of elliptic functions in (\ref{eq:1.41}) has the following long-period limit:
\begin{align}
\label{eq:1.44}
&k \to
\l\{
\begin{array}{ll}
\ds 1 &  \text{if}~ \omega >c^2/4,\\[3pt]
\ds \frac{1}{\sqrt{2}} & \text{if}~ \omega =c^2/4 ~\text{and}~ c>0,
\end{array}
\r.
\end{align}
as $L\to\infty$ (see Lemma \ref{lem:3.6}). The difference of long-period limit of modulus is essential in order that exact periodic solutions yield two types\footnote{From the explicit formulae (\ref{eq:1.7}), the soliton for the case $\omega >c^2/4$ has exponential decay and the soliton for the massless case has algebraic decay. However, exact periodic solutions are represented by the same formula as (\ref{eq:1.41}) in both two cases.} of the solitons on the whole line. Interestingly, indeterminate forms in the long-period limit appear in the massless case. 

To compute the long-period limit, it is often useful to use the maximum value $\sqrt{\eta_3}$ of $\Phi_{\omega ,c}^L$ as a parameter instead of the length of torus $L$. This idea can be seen in \cite{ABS06, A07, AN09}. To apply this idea to our setting, we need to prove
\begin{align}
\label{eq:1.45}
\sqrt{\eta_3} \to \Phi_{\omega ,c}(0) \iff L\to\infty .
\end{align}
The relation (\ref{eq:1.45}) follows from the monotonicity of the functions $\eta_3\mapsto k$ and $\eta_3 \mapsto T_{\Phi_{\omega ,c}^L}$ (see Proposition \ref{prop:3.5} and Proposition \ref{prop:3.7}), where $T_{\Phi_{\omega ,c}^L}$ is the fundamental period of $\Phi_{\omega ,c}^L$. We note that the proofs of these monotonicity are more delicate compared with one-parameter case discussed in previous works. In our proofs, the curve (\ref{eq:1.12}) corresponding to the scaling
is effectively used to derive the detailed properties including the monotonicity.

Next, we study the regularity of the convergence of exact periodic traveling wave solutions in the long-period limit. We can improve the pointwise convergence in Theorem \ref{thm:1.1} as follows.
\begin{theorem}
\label{thm:1.2}
Let $(\omega ,c) \in \R^2$ satisfy \textup{(\ref{WC})}. If for $(\omega ,c)\in\R^2$ we take sufficiently large $L$ such that $L_0<L$, then $\Phi_{\omega ,c}^L$ is well-defined by \textup{(\ref{eq:1.41})}. Then, we have
\begin{align}
\label{eq:1.46}
\lim_{L\to \infty} \| \Phi_{\omega ,c}^L -\Phi_{\omega ,c}\|_{H^m ([-L,L])} =0
\end{align}
for any $m \in \Z_{\geq 0}$.
\end{theorem}
\begin{theorem}
\label{thm:1.3}
Let $(\omega ,c) \in \R^2$ and $\Phi_{\omega ,c}^L$ in the same assumption as Theorem \textup{\ref{thm:1.2}}. Then, we have
\begin{align}
\label{eq:1.47}
\lim_{L\to \infty} \| \Phi_{\omega ,c}^L -\Phi_{\omega ,c}\|_{C^m([-L,L])} =0
\end{align}
for any $m \in \Z_{\geq 0}$.
\end{theorem}
\begin{remark}
\label{rem:1.3}
We note that Theorem~\ref{thm:1.3} is not proved directly from Theorem~\ref{thm:1.2} by using the Sobolev embedding $H^m([-L,L]) \subset C^k([-L,L]) ~(k<m)$, because constants in the Sobolev inequality depend on size of the interval $2L$. 
\end{remark}
\begin{remark}
\label{rem:1.4}
We can replace $\Phi_{\omega ,c}^L$ [resp. $\Phi_{\omega ,c}$] by $\varphi_{\omega ,c_L}^L$ [resp. $\varphi_{\omega ,c}$] in both (\ref{eq:1.46}) and (\ref{eq:1.47}) if we take $c_L \in \frac{2\pi}{L}\Z$ such that $c_L\to c$ as $L\to\infty$. Especially, we obtain the uniform bound of periodic traveling wave solutions as
\begin{align}
\label{eq:1.48}
\sup_{L_0<L<\infty} \| u_{\omega ,c_L}^L\|_{L^{\infty} (\R ,H^m(\T_{2L}))} <\infty 
\end{align}
for any $m\in\Z_{\geq 0}$, where $u_{\omega ,c_L}^L$ is defined by (\ref{eq:1.42}). 
\end{remark}
To the best of our knowledge, the regularity results in Theorem \ref{thm:1.2} and Theorem \ref{thm:1.3} are new even if we restrict the one-parameter case ($\omega >0$ and $c=0$). For the proof of Theorem \ref{thm:1.2} and Theorem \ref{thm:1.3}, to prove the $L^2$-convergence in the long-period limit is the key step.  First, we show that the mass of exact periodic solutions is exactly same as the mass of the solitons in the long-period limit (see Theorem \ref{thm:4.1}). Here again, the difference between the case $\omega >c^2/4$ and the massless case appears. We need to do a delicate calculation of elliptic integrals in this step. Next, by combining pointwise convergence in Theorem \ref{thm:1.1} and the Br\'{e}zis-Lieb lemma, we obtain $L^2$-convergence. Since $\Phi_{\omega ,c}^L$ and $\Phi_{\omega ,c}$ satisfy the same elliptic equation, we can obtain $H^2$-convergence from $L^2$-convergence and the equation. Especially, we obtain $L^{\infty}$-convergence from $H^1$-convergence. The proof of $L^{\infty}$-convergence here is related to the proof of the Sobolev inequality, but we need to calculate the dependence of the size $L$ more carefully. The rest of the proof is done by a standard bootstrap argument. We note that the detailed properties of exact periodic solutions are used throughout the proof. 

We remark that one can apply our approach to periodic wave solutions of other type of dispersive equations such that KdV, mKdV and cubic NLS in previous results (see \cite{A09} and references therein). 
\begin{remark}
\label{rem:1.5}
If we consider the periodic gauge transformed solution
\begin{align}
\label{eq:1.49}
v_{\omega ,c_L}^L := \scG_{-\frac{1}{4}}(u_{\omega ,c_L}^L) (t, x-\frac{1}{2}\mu t ),
\end{align}
then $v_{\omega ,c_L}^L$ satisfies the following equation:
\begin{align*}
i\del_t v +\del_x^2 v +i|v|^2\del_x v +e_L (v) =0,
\end{align*}
where
\begin{align*}
e_L (v) &:= \psi (v)v +\frac{1}{4}\mu |v|^2v, \\
\psi (v)&:= -\frac{1}{8L} \int_{0}^{2L} \l( 2\im (\overline{v}\del_x v) (t, \theta ) +\frac{1}{8} |v|^4 (t ,\theta ) \r) d\theta +\frac{1}{16}\mu^2 ,
\end{align*}
and $\mu =\frac{1}{2L} \| v\|_{L^2(\T_{2L})}^2$. From the uniform bound (\ref{eq:1.48}) and formula of the error term, we deduce that
\begin{align}
\label{eq:1.50}
\| e_L (v_{\omega ,c}^L)\|_{L^{\infty} (\R , H^m(\T_{2L}))} \to 0
\end{align}
as $L\to\infty$ for any $m\in\Z_{\geq 0}$. This means that the solution $v_{\omega ,c_L}^L$ gives the main part of $2L$-periodic traveling wave solutions of (\ref{DNLS}) which yield the solitons in the long-period limit, at least when $L$ is sufficiently large. One can apply a similar discussion to the equation (\ref{eq:1.1}) on $\T_{2L}$. 
\end{remark}


\subsection{Related problems and remarks} 
Compared with the solitons on the whole line, it is natural to consider that the periodic traveling waves defined by (\ref{eq:1.42}) belong to the ground states, but the rigorous proof has not been obtained yet. Variational characterizations on a torus have different difficulties from the whole line case. Since a torus is compact, the existence of a minimizer for the problems is easily obtained. However, the identification of this minimizer is a delicate problem since the elliptic equation (\ref{ELL}) has rich structure of solutions compared with the one on the whole line. This problem is also related to uniqueness of ground states. 
Recently, variational characterizations of periodic waves for cubic NLS were obtained in \cite{GLT16}, but the problems in our setting are more delicate. 

The stability/instability of the periodic traveling waves is a natural problem as a next step. First, we note that $\varphi_{\omega ,c_L}^L$ satisfies the equation (\ref{eq:1.38}) on $\T_{2L}$, which is equivalent that
\begin{align}
\label{eq:1.51}
{\cS}_{\omega ,c_L} ' (\varphi_{\omega ,c_L}^L ) =0,
\end{align}
where 
\begin{align*}
{\cS}_{\omega ,c} (\varphi )={\cE}(\varphi )+\frac{\omega}{2} {\cM} (\varphi )+\frac{c}{2} {\cP} (\varphi ).
\end{align*} 
The relation (\ref{eq:1.51}) is important when one considers the problems of both variational characterization and stability. There are several difficulties when one considers the stability/instability problem in our setting. We note that the equation (\ref{DNLS1}) can not be rewritten as the Hamiltonian form by using the energy functional as (\ref{eq:1.4}).
The lack of Hamiltonian structure causes the delicate problems when one considers the stability/instability problem; see \cite{GW95} for partial results on the stability.\footnote{In \cite{GW95} they consider the stability problem on the whole line in the setting which can not be rewritten as the Hamiltonian form as (\ref{eq:1.4}).}
To prove stability or instability of solitons, it is important to calculate second derivatives. However, since we only take $c_L \in \frac{2\pi}{L}\Z$ as a discrete value, this gives the difficulty of differential calculation of ${\cS}_{\omega ,c_L} (\varphi_{\omega ,c_L}^L )$. We recall that Colin and Ohta \cite{CO06} proved orbital stability of solitons (\ref{eq:1.5}) by showing that the matrix $d'' (\omega ,c)$ has one positive eigenvalue, where $d(\omega ,c)$ is defined by
\begin{align}
\label{eq:1.52}
d(\omega ,c ) := S_{\omega ,c} (\phi_{\omega ,c}) .
\end{align}
We note that when $\omega >c^2/4$ and $c>0$ we have
\begin{align}
\label{eq:1.53}
\del_{\omega}^2 d(\omega ,c) &=\frac{1}{2}\del_{\omega} M(\phi_{\omega ,c}) = -\frac{c}{\omega\sqrt{4\omega -c^2}}<0, \\
\label{eq:1.54}
\del_{c}^2 d(\omega ,c) &=\frac{1}{2}\del_{c} P (\phi_{\omega ,c}) = -\frac{c}{\sqrt{4\omega -c^2}}<0. 
\end{align}
If one considers the solitons as a one-parameter $\omega\mapsto\phi_{\omega ,c}$ or $c\mapsto\phi_{\omega ,c}$, (\ref{eq:1.53}) and (\ref{eq:1.54}) seem to indicate that the solitons are unstable, but actually they are stable. This means that the calculation as a one-parameter is not enough to fix the stability problems, and shows one of the deep structure of a two-parameter family of solitons. 

Although there are several difficulties on the stability/instability problems as above, it is important to study these problems in understanding further properties of exact periodic traveling wave solutions and related dynamics.
We refer to \cite{A07, A09, ABS06, AN08, AN09, GH07a, GH07b, GLT16} for the studies on the stability/instability of the periodic profiles.
The author hopes that our results in this paper would provide further insight on the dynamics for the derivative nonlinear Schr\"{o}dinger equation.


\subsection{Organization of the paper}
The rest of this paper is organized as follows. In Section \ref{sec:2} we recall the definition and basic properties of elliptic functions and elliptic integrals. In Section \ref{sec:3} we discuss construction and fundamental properties of exact periodic traveling wave solutions, and give a proof of Theorem \ref{thm:1.1}. In Section \ref{sec:4}, we discuss the regularity of the convergence in the long-period limit and prove Theorem \ref{thm:1.2} and Theorem \ref{thm:1.3}. In Appendix \ref{sec:A} and \ref{sec:B}, we prove monotonicity of the modulus and the fundamental period.


\section{Preliminaries}
\label{sec:2}%
Here, we recall the definitions and some basic properties of elliptic functions and elliptic integrals. We refer the reader to \cite{BF71, L89} for more details. Given $k \in (0,1)$, the incomplete elliptic integral of the first kind is defined by
\begin{align*}
u= F(\varphi ,k) := \int_{0}^{\varphi} \frac{d\theta}{\sqrt{1-k^2\sin^2 \theta}}.
\end{align*}
The Jacobi elliptic functions are defined through the inverse function of $F(\cdot ,k)$ by
\begin{align*}
\mathrm{sn} (u;k):=\sin \varphi ,~\mathrm{cn}(u;k) :=\cos \varphi ,~
\mathrm{dn}(u;k) :=\sqrt{1-k^2\mathrm{sn}^2 (u;k)} .
\end{align*}
The complete elliptic integral of the first kind is defined by
\begin{align*}
K=K(k):= F\l( \frac{\pi}{2} ,k\r) .
\end{align*}
The functions $\mathrm{sn}$, $\mathrm{cn}$ and $\mathrm{dn}$ have a real fundamental period, namely, $4K$, $4K$, and $2K$, respectively. We note that
\begin{align}
\label{eq:2.1}
K(k) \to
\l\{
\begin{array}{ll}
\frac{\pi}{2} & \text{as}~k\to 0,\\[3pt]
 \infty & \text{as}~k\to 1.
\end{array}
\r.
\end{align}
More specifically, when $k\to 1$, the function $K(k)$ has the following asymptotic behavior: %
\begin{align}
\label{eq:2.2}
\lim_{k\to 1} \l( K(k)-\log \frac{4}{k'}\r) =0,
\end{align}
where the complementary modulus $k'$ is define by
\begin{align*}
k' :=\sqrt{1-k^2}.
\end{align*}
Elliptic functions have the following extremal formulae:
\begin{align}
\label{eq:2.3}
\l\{
\begin{array}{l}
\ds \mathrm{sn} (u ;0)=\sin u, ~\mathrm{cn} (u;0) =\cos u, ~\mathrm{dn} (u ;0) \equiv 1, \\[3pt]
\ds \mathrm{sn} (u ;1)=\tanh u, ~\mathrm{cn} (u;1) =\mathrm{dn} (u ;1) =\mathrm{sech}\,u.
\end{array}
\r.
\end{align}
This shows that elliptic functions bridge the gap between trigonometric and hyperbolic functions. 

The incomplete elliptic integral of the second kind is defined by
\begin{align*}
E(\varphi ,k) := \int_{0}^{\varphi} \sqrt{1-k^2\sin^2 \theta}d\theta .
\end{align*}
The complete elliptic integral of the second kind is defined by
\begin{align*}
E=E(k):= E\l( \frac{\pi}{2} ,k\r) .
\end{align*}
We define by
\begin{align*}
&K'=K'(k):=K(k'), \\
&E'=E'(k):= E(k').
\end{align*}
Then, we have the following Legendre relation
\begin{align}
\label{eq:2.4}
EK'+E'K-KK' =\frac{\pi}{2}\quad\text{for all}~ k \in (0,1).
\end{align}
\section{Existence of exact periodic traveling waves}
\label{sec:3}
\subsection{Construction of exact solutions}
\label{sec:3.1}
We consider the elliptic equation (\ref{ELL}) on $\T_{2L}$. Set $\psi =\Phi ^2$. By multiplying the equation (\ref{ELL}) by $\Phi'$ and integrating, $\psi$ satisfies the following equation
\begin{align}
\label{eq:3.1}
[\psi ']^2 &= -\frac{1}{4}\psi^4 +c \psi^3 +4\l( \omega - \frac{c^2}{4}\r) \psi^2 +8C_{\psi}\psi ,
\end{align}
where $C_{\psi}$ is a constant of integration. The formula (\ref{eq:3.1}) can be rewritten as
\begin{align}
\label{eq:3.2}
[\psi ']^2 = \frac{1}{4} P_{\psi} (\psi ),
\end{align}
where the polynomial $P_\psi$ is defined by
\begin{align*}
P_{\psi} (t) &=  -t^4+4ct^3 +16 \l( \omega -\frac{c^2}{4}\r) t^2 +32C_{\psi}t \\
&=t(t-\eta_{1})(t -\eta_{2}) (\eta_{3} -t).
\end{align*}
Here, $\eta_1 , \eta_2$ and $\eta_3$ are roots of the polynomial $P_{\psi}$ satisfying
\begin{align}
\label{eq:3.3}
\l\{
\begin{array}{l}
\eta_1 +\eta_2 +\eta_3 =4c, \\
\eta_2 \eta_3 +\eta_1 \eta_3 +\eta_1 \eta_2 =-16 \l( \omega - \frac{c^2}{4}\r) ,\\
\eta_1 \eta_2 \eta_3 =32C_{\psi}.
\end{array}
\r.
\end{align}
Since we are interested in the positive solution, we may set $0 < \eta_{2}<\eta_{3}$. We note that $\eta_3$ [resp. $\eta_2$] is the maximum [resp. minimum] value of $\psi$ by (\ref{eq:3.2}). By (\ref{eq:3.3}) and (\ref{WC}), $\eta_1$ must be negative.  By invariance of translations, we may assume that $\psi (0)=\eta_3$ and $\psi '(0)=0$. From uniqueness of the ordinary differential equation and the equation (\ref{ELL}), $\psi$ is even. Since we want to construct single-bump solutions, we may assume that $\psi (L)=\eta_2$.
Therefore, it is enough to consider the equation (\ref{eq:3.2}) on $[0, L]$. Since $\psi ' (x) <0$ when $0 <x < L$, integrating both sides of (\ref{eq:3.2}) over $[0,x]$ yields that
\begin{align*}
- \int_{0}^x \frac{\psi ' (y)}{\sqrt{P_{\psi} (\psi (y))}} dy = \frac{1}{2}x .
\end{align*}
Changing variables $t = \psi (x)$ in the integral implies that
\begin{align}
\int_{\psi (x)}^{\eta_3} \frac{dt}{\sqrt{t(\eta_3 -t)(t-\eta_2 )(t-\eta_1 )}} = \frac{1}{2}x .
\end{align}
Applying the formula 257.00 in \cite{BF71}, we conclude that
\begin{align}
\label{eq:3.5}
\psi (x) =\frac{\eta_3 (\eta_2 -\eta_1 )+(\eta_3 -\eta_2 )\eta_1 \mathrm{sn}^2\l( \frac{x}{2g} ;k\r) }{(\eta_2 -\eta_1 )+(\eta_3 -\eta_2 ) \mathrm{sn}^2 \l( \frac{x}{2g} ; k\r)} ,
\end{align}
where
\begin{align}
\label{eq:3.6}
k^2 &=\frac{-\eta_{1}(\eta_3-\eta_2 )}{\eta_3 (\eta_2 -\eta_1 )}, \\
\label{eq:3.7}
g&=\frac{2}{\sqrt{\eta_3 (\eta_2 -\eta_1 )}} .
\end{align}
We note that $0<k^2<1$ from the inequality $\eta_1 <0<\eta_2 <\eta_3$. By using the expression of $k$, the formula (\ref{eq:3.5}) can be rewritten as
\begin{align}
\label{eq:3.8}
\psi (x) =\eta_3 \l[ \frac{\mathrm{dn}^2 \l( \frac{x}{2g} ; k\r)}{1+ \beta^2 \mathrm{sn}^2 \l( \frac{x}{2g} ; k\r)} \r] ,
\end{align}
with $\beta^2 =-\eta_3 k^2/\eta_1 >0$. From the fundamental periods of $\mathrm{sn}$ and $\mathrm{dn}$, the fundamental period $T_\psi$ of $\psi$ is given by
\begin{align}
\label{eq:3.9}
T_{\psi}= 4g K(k) = \frac{8}{\sqrt{\eta_3 (\eta_2 -\eta_1 )}} K(k).
\end{align}
Since we assume $\psi$ is the single-bump solution, we obtain
\begin{align}
\label{eq:3.10}
2L =T_{\psi}= \frac{8}{\sqrt{\eta_3 (\eta_2 -\eta_1 )}} K(k).
\end{align}
Substituting the first equation in (\ref{eq:3.3})
\begin{align}
\label{eq:3.11}
\eta_1 =4c -\eta_2 -\eta_3
\end{align}
into the second equation in (\ref{eq:3.3}), we obtain
\begin{align}
\label{eq:3.12}
\eta_2^2 + \eta_3^2 +\eta_2 \eta_3 -4c (\eta_2 +\eta_3 )-16 \l( \omega -\frac{c^2}{4}\r) =0.
\end{align}
From (\ref{eq:3.11}) and (\ref{eq:3.12}), $\eta_1$ and $\eta_2$ have expressions as functions of $\eta_3, \omega$ and $c$ as
\begin{align}
\label{eq:3.13}
\eta_1 &=\frac{-\eta_3 +4c -\sqrt{A}}{2}, \\
\label{eq:3.14}
\eta_2 &=\frac{-\eta_3 +4c +\sqrt{A}}{2},
\end{align} 
where $A$ is defined by
\begin{align}
\label{eq:3.15}
A=A(\eta_3 ):=64\omega -3\eta_3^2 +8c\eta_3 .
\end{align}
The following two extreme cases can be considered;
\begin{enumerate}[(i)]
\setlength{\itemsep}{3pt}

\item $\eta_2 =\eta_3=:\alpha_0$.

\item $\eta_2 =0$, $\eta_3=:\alpha_1$.
\end{enumerate}
The case (i) corresponds to the constant solution of (\ref{ELL}).
The case (ii) corresponds to the long-period limit as discussed in detail later. From the equation (\ref{eq:3.12}), we obtain that
\begin{align*}
\alpha_0 &= \frac{1}{3} \l( 4c+\sqrt{48\omega +4c^2}\r) , \\
\alpha_1 &=4\sqrt{\omega}+2c .
\end{align*}
It is worthwhile to note that $\alpha_1 =\Phi_{\omega ,c}^2(0)$, where $\Phi_{\omega ,c}$ is defined by (\ref{eq:1.7}). 
\subsection{Fundamental properties of exact solutions}
In this subsection, we investigate detailed relation between parameters defined in Section \ref{sec:3.1}. 
For the convenience of calculation, we introduce the following notations. When $(\omega ,c)$ satisfies (\ref{WC}), we can write
\begin{align*} 
c=2s\sqrt{\omega}
\end{align*} 
for $\omega >0$ and some  $s \in (-1, 1]$. The case $s=1$ corresponds to the massless case. By using this notation, $\alpha_0$ and $\alpha_1$ are rewritten as
\begin{align*}
\alpha_0 &=\frac{4}{3}\l( 2s+\sqrt{3+s^2}\r) \sqrt{\omega}, \\
\alpha_1 &=4(1+s)\sqrt{\omega}.
\end{align*}
Set $\beta_0 =\frac{\alpha_0}{4\sqrt{\omega}}$ and $\beta_1=\frac{\alpha_1}{4\sqrt{\omega}}$. We have
\begin{align}
\label{eq:3.16}
\beta_0 &=\beta_0 (s)=\frac{1}{3}\l( 2s+\sqrt{3+s^2}\r) ,\\
\label{eq:3.17}
\beta_1 &=\beta_1 (s)=1+s
\end{align}
for $-1<s\leq 1$. We begin with the following lemma.
\begin{lemma}
\label{lem:3.1}
Let $(\omega ,c)$ satisfy \textup{(\ref{WC})}. Then, we have $0<\alpha_0 <\alpha_1$. 
\end{lemma}
\begin{proof}
From the definition, we note that
\begin{align*}
0<\alpha_0 <\alpha_1 
\iff 0<\beta_0 < \beta_1 .
\end{align*}
First, we prove $\beta_ 0 >0$. This is trivial from the definition (\ref{eq:3.16}) when $0\leq s\leq 1$. When $s<0$, we have
\begin{align*}
\beta_0 >0 &\iff -2s <\sqrt{3+s^2} \\
&\iff 0<3(1-s^2).
\end{align*}
The last inequality holds when $-1<s<0$.

Next, we prove $\beta_0 <\beta_1$. When $-1<s\leq 1$, we have
\begin{align*}
\beta_0  < \beta_1 
&\iff \frac{1}{3}\l( 2s+\sqrt{3+s^2}\r) <1+s \\
&\iff 0<6(s+1).
\end{align*}
The last inequality holds when $-1<s\leq 1$. This completes the proof.
\end{proof}
\begin{figure}[t]
\begin{center}
\includegraphics[width=10cm]{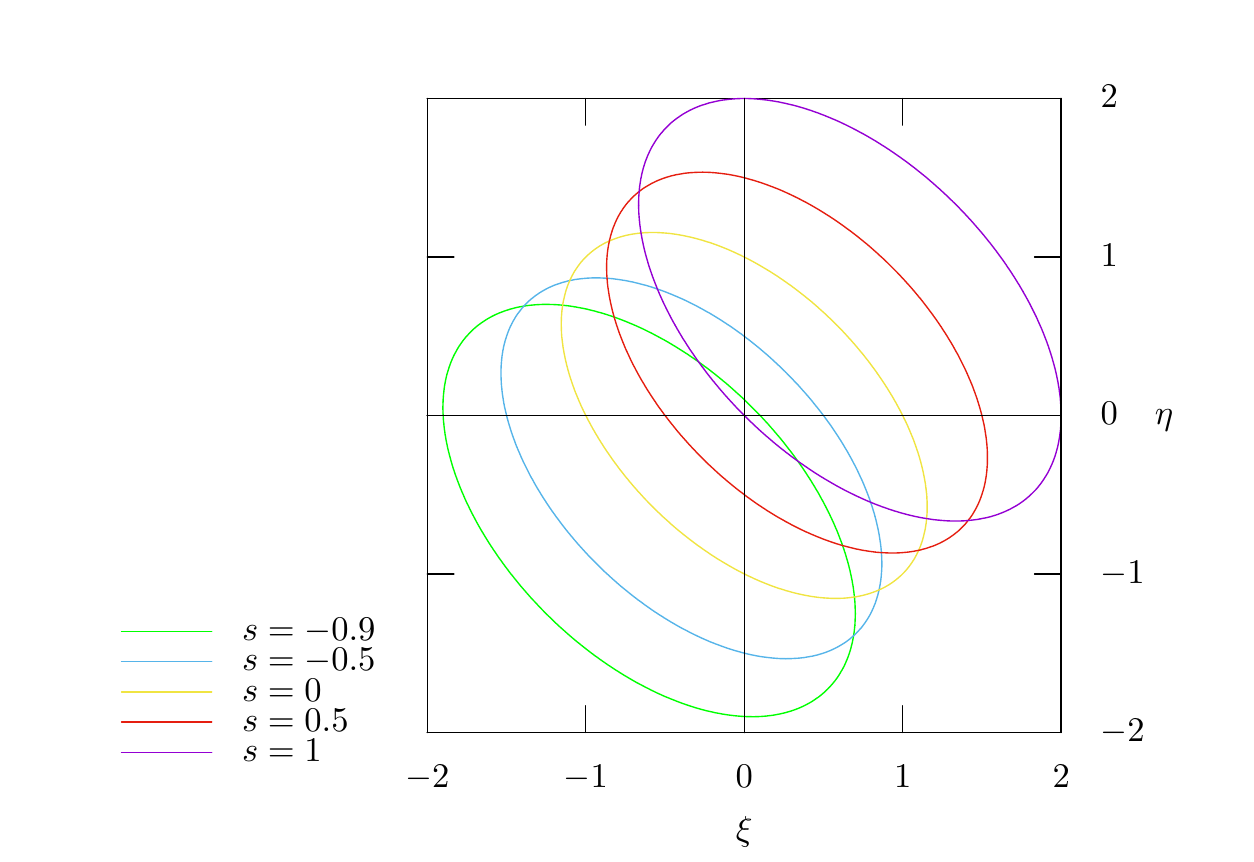}
\caption{The ellipse (\ref{eq:3.19}) for several values of $s$. Note that the ellipse moves toward upper right when one changes the parameter $s$ from $-1$ to $1$.}
\label{fig:1}
\end{center}
\end{figure}
We recall that $(\eta_2 , \eta_3)$ satisfies the constraint condition (\ref{eq:3.12}). Set 
\begin{align}
\label{eq:3.18}
\xi =\frac{\eta_2}{4\sqrt{\omega}}, \quad \eta =\frac{\eta_3}{4\sqrt{\omega}}
\end{align}
Substituting (\ref{eq:3.18}) and $c=2s\sqrt{\omega}$ into (\ref{eq:3.12}), the equation (\ref{eq:3.12}) is equivalent that
\begin{align}
\label{eq:3.19}
(\xi -s)^2+(\eta -s)^2+\xi\eta =1+s^2,
\end{align}
where $-1<s\leq 1$. The equation (\ref{eq:3.19}) represents the ellipse as in Figure \ref{fig:1}. Note that $(\beta_0 , \beta_0)$ corresponds to a intersection point between line $\eta =\xi$ and ellipse (\ref{eq:3.19}), and that $(0, \beta_1)$ corresponds to a intersection point between line $\xi=0$ and ellipse (\ref{eq:3.19}). Since we assumed that $0<\eta_2 <\eta_3$ in Section \ref{sec:3.1}, it follows that
\begin{align}
\label{eq:3.20}
\alpha_0 < \eta_3 <\alpha_1 ,
\end{align}
or equivalently
\begin{align}
\label{eq:3.21}
\beta_0 < \eta <\beta_1 .
\end{align}

We can prove positivity of $A$ defined by (\ref{eq:3.15}) under the condition (\ref{eq:3.20}).
\begin{lemma}
\label{lem:3.2}
Let $(\omega ,c)$ satisfy \textup{(\ref{WC})} and let $\eta_3$ satisfy \textup{(\ref{eq:3.20})}. Then, we have $A=A(\eta_3 )>0$.
\end{lemma}
\begin{proof}
By using $c=2s\sqrt{\omega}$ and $\eta_3 =4\sqrt{\omega}\eta$, we can rewrite $A$ as
\begin{align*}
A&=64\omega -3\eta_3^2 +8c\eta_3 \\
&=16\omega (-3\eta^2 +4s\eta +4).
\end{align*}
We define the function $f_s$ by 
\begin{align}
\label{eq:3.22}
f_s(\eta ) :=-3\eta^2 +4s\eta +4
\end{align}
for $-1<s\leq 1$. A positive zero of $f_s(\eta )$ is given by
\begin{align}
\label{eq:3.23}
\gamma =\frac{2}{3}\l( s+\sqrt{s^2+3}\r) .
\end{align}
We obtain $\beta_1\leq\gamma$ for $-1<s\leq 1$. Indeed, we have
\begin{align*}
\beta_1 \leq \gamma &\iff 1+s \leq \frac{2}{3}\l( s+\sqrt{s^2+3}\r)\\
&\iff 0 \leq 3(s-1)^2.
\end{align*}
The last inequality means that $\beta_1 =\gamma$ when $s=1$ and $\beta_1 <\gamma$ otherwise. Since $f_s(0) =4>0$ and $0<\beta_0 <\beta_1$, this implies that $f_s (\eta ) >0$ for $-1 <s\leq 1$ and $\beta_0 <\eta <\beta_1$. This completes the proof.
\end{proof}
\begin{remark}
\label{rem:3.1}
From the proof of Lemma \ref{lem:3.2} above, we also deduce that $A( \alpha_0)$ is a positive constant depending on $(\omega ,c)$.
\end{remark}
From Figure \ref{fig:1}, one can observe that $\eta_2$ decreases from $\alpha_0$ to $0$ when one changes $\eta_3$ from $\alpha_0$ to $\alpha_1$.  We can prove this result rigorously.
\begin{lemma}
\label{lem:3.3} 
Let $(\omega ,c)$ satisfy \textup{(\ref{WC})}. Then, the function $(\alpha_0, \alpha_1) \ni\eta_3\mapsto\eta_2\in (0, \alpha_0)$ is a strictly decreasing function.
\end{lemma}
\begin{proof}
It is enough to prove that the function $(\beta_0, \beta_1) \ni\eta\mapsto\xi\in (0, \beta_0)$ is a strictly decreasing function. From (\ref{eq:3.14}) and (\ref{eq:3.22}), we have
\begin{align*}
\xi =\frac{1}{2}\l( -\eta +2s+\sqrt{f_s(\eta )}\r) .
\end{align*}
For $-1 <s\leq 1$ and $\beta_0 <\eta <\beta_1$, we have
\begin{align*}
\frac{d\xi}{d\eta}
&=-\frac{1}{2}-\frac{1}{4\sqrt{f_s(\eta )}}(6\eta -4s) \\
&<-\frac{1}{2}-\frac{1}{2\sqrt{f_s(\eta )}} \l( 3\cdot \frac{2s}{3} -2s\r) =-\frac{1}{2}<0,
\end{align*}
where we used the following inequality:
\begin{align*}
\frac{2s}{3} <\frac{2s+\sqrt{3+s^2}}{3}=\beta_0 <\eta .
\end{align*}
This completes the proof.
\end{proof}
Next, we discuss the change of the parameters when we take the limit $\eta_{3} \to \alpha_1$ (or equivalently $\eta\to\beta_1$). We begin with the following lemma.
\begin{lemma}
\label{lem:3.4}
Let $(\omega ,c)$ satisfy \textup{(\ref{WC})}. Then, we have
\begin{align}
\label{eq:3.24}
\eta_1 \to -4\sqrt{\omega}+2c , ~\eta_2 \to 0,~\frac{1}{2g} \to \frac{\sqrt{4\omega -c^2}}{2}
\end{align}
as $\eta_3\to\alpha_1$.
\end{lemma}
\begin{proof}
We note that
\begin{align*}
A(\alpha_1 ) &=64\omega -3\alpha_1 +8c\alpha_1 \\
&=(4\sqrt{\omega}-2c)^2 .
\end{align*}
From the expressions (\ref{eq:3.13}) and (\ref{eq:3.14}), we have
\begin{align*}
\lim_{\eta_3\to\alpha_1}\eta_1 &= \frac{-\alpha_1 +4c-\sqrt{A(\alpha_1 )}}{2}\\
&=\frac{-4\sqrt{\omega}+2c -\l( 4\sqrt{\omega}-2c\r)}{2} =-4\sqrt{\omega}+2c,\\
\lim_{\eta_3\to\alpha_1}\eta_2 &= \frac{-\alpha_1 +4c+\sqrt{A(\alpha_1 )}}{2}\\
&=\frac{-4\sqrt{\omega}+2c +\l( 4\sqrt{\omega}-2c\r)}{2} =0.
\end{align*}
Note that the limit of $\eta_2$ compatible with the definition of $\alpha_1$. From the expression (\ref{eq:3.7}) and the limits of $\eta_2$ and $\eta_3$, we have
\begin{align*}
\lim_{\eta_3\to\alpha_1}\frac{1}{2g} &=\lim_{\eta_3\to\alpha_1}\frac{ \sqrt{\eta_3 (\eta_2 -\eta_1)} }{4} \\
&=\frac{ \sqrt{ (4\sqrt{\omega}+2c)(4\sqrt{\omega}-2c) } }{4}=\frac{ \sqrt{4\omega -c^2} }{2}.
\end{align*}
This completes the proof.
\end{proof}
It is more delicate to calculate the limit of modulus $k$ of elliptic functions as $\eta_3\to\alpha_1$. First, we rewrite $k^2$ defined by (\ref{eq:3.6}) as a function of $\eta$. From (\ref{eq:3.13}) and (\ref{eq:3.14}), we have
\begin{align*}
\eta_3 (\eta_2 -\eta_1 )&=4\sqrt{\omega}\eta\cdot\sqrt{A}=16\omega \eta \sqrt{f_s(\eta )}.
\end{align*}
Since
\begin{align*}
\eta_3 -\eta_2 =\frac{3\eta_3 -4c-\sqrt{A}}{2},
\end{align*}
we have
\begin{align*}
4\l( -\eta_1 (\eta_3-\eta_2 )\r) &=(3\eta_3 -4c-\sqrt{A}) (\eta_3 -4c+\sqrt{A}) \\
&=32\omega \l( 3\eta^2 +(\sqrt{f_s(\eta )}-6s)\eta +2(s^2-1)\r) .
\end{align*}
Hence, we have the expression of $k^2$ as
\begin{align}
\label{eq:3.25}
k^2 &=\frac{-\eta_{1}(\eta_3-\eta_2 )}{\eta_3 (\eta_2 -\eta_1 )} \\
&=\frac{3\eta^2 +(\sqrt{f_s(\eta )}-6s)\eta +2(s^2 -1)}{2\eta\sqrt{f_s(\eta )}}.\notag 
\end{align}
for $\beta_0<\eta <\beta_1$ and $-1<s\leq 1$. By using the expression of (\ref{eq:3.25}), we can prove the monotonicity of modulus $k$ of elliptic functions.
\begin{proposition}
\label{prop:3.5}
Let $(\omega ,c)$ satisfy \textup{(\ref{WC})}. Then, the function $(\beta_0 ,\beta_1) \ni\eta\mapsto k(\eta) \in (0,1)$ is a strictly increasing function.
\end{proposition}
We will prove Proposition \ref{prop:3.5} in Appendix \ref{sec:A}. The limits of $k$ and $\beta^2$ are given by the following lemma.
\begin{lemma}
\label{lem:3.6}
Let $(\omega ,c)$ satisfy \textup{(\ref{WC})}. Then, we have
\begin{align}
\label{eq:3.26}
&k \to
\l\{
\begin{array}{ll}
\ds 1 &  \text{if}~ \omega >c^2/4,\\[3pt]
\ds \frac{1}{\sqrt{2}} & \text{if}~ \omega =c^2/4 ~\text{and}~ c>0,
\end{array}
\r.
\\[3pt]
\label{eq:3.27}
&\beta^2\to
\l\{
\begin{array}{ll}
\ds \frac{2\sqrt{\omega}+c}{2\sqrt{\omega}-c} &  \text{if}~ \omega >c^2/4,\\[7pt]
\ds \infty & \text{if}~ \omega =c^2/4 ~\text{and}~ c>0.
\end{array}
\r.
\end{align}
as $\eta_3\to\alpha_1$.
\end{lemma}
\begin{proof}
Case 1: $\omega >c^2/4$. By Lemma \ref{lem:3.4} we note that
\begin{align*}
\eta_1 \to -4\sqrt{\omega}+2c <0, \quad \eta_2 \to 0
\end{align*}
as $\eta_3 \to \alpha_1$. From the definitions of $k^2$ and $\beta$, we obtain
\begin{align*}
\lim_{\eta_3\to\alpha_1}k^2 &=\lim_{\eta_3\to\alpha_1}\frac{-\eta_{1}(\eta_3-\eta_2 )}{\eta_3 (\eta_2 -\eta_1 )} \\
&=\frac{(4\sqrt{\omega}-2c)\cdot \alpha_1}{\alpha_1 (4\sqrt{\omega}-2c) }=1, 
\end{align*}
and
\begin{align*}
\lim_{\eta_3\to\alpha_1} \beta^2 &=\lim_{\eta_3\to\alpha_1} -\frac{\eta_3 k^2}{\eta_1}\\
&=\frac{2\sqrt{\omega}+c}{2\sqrt{\omega}-c} .
\end{align*}
Case 2: $\omega =c^2/4$ and $c>0$. Since in this case
\begin{align*}
\eta_1 \to -4\sqrt{\omega}+2c=0, \quad \eta_2 \to 0
\end{align*}
as $\eta_3 \to \alpha_1$, the above calculation does not work. Since $s=1$ in this case, from the expression (\ref{eq:3.25}) we have
\begin{align}
\label{eq:3.28}
k^2(\eta )&=\frac{3\eta +(\sqrt{f_1(\eta )}-6)}{2\sqrt{f_1(\eta )}}\\
&=\frac{3(\eta -2)}{2\sqrt{f_1(\eta )}} +\frac{1}{2} ,\notag
\end{align}
where $\beta_0 (1) <\eta <\beta_1 (1)=2$. Note that
\begin{align}
\label{eq:3.29}
f_1(\eta )=-(3\eta^2 -4\eta -4)=-(3\eta +2 ) (\eta -2) .
\end{align}
From (\ref{eq:3.28}) and (\ref{eq:3.29}), we deduce that
\begin{align*}
k^2(\eta )&=-\frac{3(2-\eta )}{2\sqrt{(2-\eta )(3\eta +2)}} +\frac{1}{2}\\
&=-\frac{3\sqrt{2-\eta}}{2\sqrt{3\eta +2}}+\frac{1}{2} \\
&\longrightarrow \frac{1}{2} 
\end{align*}
as $\eta\to \beta_1$. This gives \eqref{eq:3.26} for the massless case. 
For the limit of $\beta^2$, since $\eta_1 \to 0$ as $\eta_3\to\alpha_1$, we have
\begin{align*}
\lim_{\eta_3\to\alpha_1} \beta^2 &=\lim_{\eta_3\to\alpha_1} -\frac{\eta_3 k^2}{\eta_1} =\infty .
\end{align*}
This completes the proof.
\end{proof}
The fundamental period $T_{\psi}$ defined by (\ref{eq:3.9}) is rewritten as
\begin{align*}
T_{\psi}(\eta_3 ) =\frac{8}{ \sqrt{\eta_3\sqrt{A(\eta_3 )}} } K(k(\eta_3 ))
\end{align*}
for $\alpha_0 <\eta_3 <\alpha_1$. Combined with Proposition \ref{prop:3.5}, we can prove the monotonicity of the fundamental period $T_{\psi}$. 
\begin{proposition}
\label{prop:3.7}
Let $(\omega ,c)$ satisfy \textup{(\ref{WC})}. Then, the function $(\alpha_0 ,\alpha_1) \ni\eta_3 \mapsto T_{\psi}(\eta_3 ) \in (0,\infty)$ is a strictly increasing function.
\end{proposition}
The proof of Proposition \ref{prop:3.7} will be given in Appendix \ref{sec:B}. From the definition (\ref{eq:3.6}) of $k$, we have
\begin{align*}
k^2 (\eta_3 ) \to 0
\end{align*}
as $\eta_3 \to\alpha_0$. Since $K(k)\to \frac{\pi}{2}$ as $k\to 0$, we have
\begin{align}
\label{eq:3.30}
T_{\psi} (\eta_3 ) \to \frac{4\pi}{ \sqrt{\alpha_0\sqrt{A(\alpha_0)}} } =:T_0(\omega ,c) 
\end{align}
as $\eta_3\to \alpha_0$. Note that $A(\alpha_0 )$ is a positive constant as described in Remark \ref{rem:3.1}. On the other hand, we have 
\begin{align}
\label{eq:3.31}
T_{\psi} (\eta_3 ) \to \infty
\end{align}
as $\eta_3 \to\alpha_1$. Indeed, when $\omega >c^2/4$, we have $k\to 1$ as $\eta_3\to\alpha_1$ by Lemma \ref{lem:3.6}. Since $K(k) \to \infty$ as $k\to 1$, (\ref{eq:3.31}) holds. When $\omega =c^2/4$ and $c>0$, we have $\eta_1, \eta_2 \to 0$ as $\eta_3 \to\alpha_1$ by Lemma \ref{lem:3.4}, and hence (\ref{eq:3.31}) holds from the definition (\ref{eq:3.9}) of $T_{\psi}$. Therefore, by (\ref{eq:3.30}), (\ref{eq:3.31}) and Proposition \ref{prop:3.7} we deduce that
\begin{align}
\label{eq:3.32}
\alpha_0 <\eta_3 <\alpha_1 \iff T_0(\omega ,c) <T_{\psi} (\eta_3 ) <\infty ,
\end{align}
and
\begin{align}
\label{eq:3.33}
\eta_3 \to \alpha_0 &\iff T_{\psi}(\eta_3 ) \to T_0 (\omega ,c), \\
\label{eq:3.34}
 \eta_3 \to \alpha_1 &\iff T_{\psi}(\eta_3 ) \to \infty .
\end{align}
The relation (\ref{eq:3.34}) means that the limit $\eta_3\to \alpha_1$ is equivalent to the long-period limit. Since $2L=T_{\psi}$, $L$ has the following constraint condition:
\begin{align}
\label{eq:3.35}
L_0(\omega ,c) <L <\infty ,
\end{align}
where $L_0(\omega ,c)$ is defined by
\begin{align}
\label{eq:3.36}
L_0 =L_0(\omega ,c):=\frac{T_0 (\omega ,c) }{2}=\frac{2\pi}{ \sqrt{\alpha_0 \sqrt{A(\alpha_0 )}} }.
\end{align}
Since by (\ref{eq:3.34}) we have
\begin{align*}
\eta_3 \to \alpha_1 \iff L\to\infty ,
\end{align*}
we can take the limit $\eta_3\to\alpha_1$ instead of the limit $L\to\infty$.

To clarify the dependence of parameters, we denote the function $\psi$ by $\psi_{\omega, c}^L$. Let $c_L \in \frac{2\pi}{L}\Z$. It is easily verified that the traveling wave
\begin{align*}
u_{\omega, c_L}=e^{i\omega t+ i\frac{c_L}{2}(x-c_Lt)} (\psi^L_{\omega, c_L})^{\frac{1}{2}}(x-c_L t)
\end{align*}
is a solution of the equation (\ref{DNLS1}). 
\subsection{Pointwise convergence in the long-period limit}
We complete the proof of Theorem \ref{thm:1.1}. Fix any $x \in \R$ and consider a large $L>0$ such that $x \in [-L, L]$. We need to divide two cases to do calculations in the long-period limit.\\
Case 1: $\omega >c^2/4$. By Lemma \ref{lem:3.4}, Lemma \ref{lem:3.6} and extremal formulae (\ref{eq:2.3}) of elliptic functions, we have
\begin{align*}
\lim_{L\to\infty} \psi^L_{\omega ,c} (x)&= \lim_{\eta_3 \to \alpha_1} \eta_3 \l[\frac{\mathrm{dn}^2 \l( \frac{x}{2g} ; k\r)}{1+ \beta^2 \mathrm{sn}^2 \l( \frac{x}{2g} ; k\r)} \r] \\
&= (4\sqrt{\omega}+2c) \l[ \frac{ \mathrm{sech}^2\l( \frac{\sqrt{4\omega -c^2}}{2}x\r) }{1+\frac{2\sqrt{\omega}+c}{2\sqrt{\omega}-c} \tanh^2 \l( \frac{\sqrt{4\omega -c^2}}{2}x\r) } \r] \\
&= \frac{2(4\omega -c^2)}{(2\sqrt{\omega}-c) \cosh^2 \l( \frac{\sqrt{4\omega -c^2}}{2}x\r) +(2\sqrt{\omega}+c)\sinh^2 \l( \frac{\sqrt{4\omega -c^2}}{2}x\r)} \\[2pt]
&=\frac{ 2(4\omega -c^2) }{2\sqrt{\omega} \cosh^2 (\sqrt{4\omega -c^2}x) -c } = \Phi_{\omega ,c}^2 (x) .
\end{align*}
Case 2: $\omega =c^2/4$ and $c>0$. Since in this case $\beta\to\infty,~\frac{1}{2g} \to 0$ as $\eta_3\to\alpha_1$, we need to calculate more carefully. We use the following relations
\begin{align}
\label{eq:3.37}
\mathrm{dn} (u;k) &=1 +O(u^2), \\
\label{eq:3.38}
\mathrm{sn} (u;k) &=u+O(u^3)
\end{align}
as $u\to 0$ (see, e.g., \cite{L89} in detail). From (\ref{eq:3.37}), we have
\begin{align}
\label{eq:3.39}
\lim_{\eta_3\to\alpha_1} \mathrm{dn} \l( \frac{x}{2g}; k\r) =1 .
\end{align}
We note that 
\begin{align}
\label{eq:3.40}
\frac{1}{4g^2} &=\frac{ \eta_3(\eta_2 -\eta_1 )}{16}\\
&=\eta\omega\sqrt{f_1(\eta )} \notag\\
&=\eta\omega\sqrt{(2-\eta )(3\eta +2)}, \notag
\end{align}
where in the last equality we used the identity (\ref{eq:3.29}). We can rewrite $\eta_1$ as
\begin{align}
\label{eq:3.41}
-\eta_1 &= \frac{ \eta_3 -4c +\sqrt{A} }{2}\\
&=2\sqrt{\omega}\sqrt{2-\eta} \l( \sqrt{3\eta +2} -\sqrt{2-\eta}\r) . \notag
\end{align}
From (\ref{eq:3.40}) and (\ref{eq:3.41}), we have
\begin{align}
\label{eq:3.42}
\beta^2 \cdot \frac{1}{4g^2} &=-\frac{\eta_3}{\eta_1}k^2 \cdot \frac{1}{4g^2}\\
&= \frac{ 2\eta^2\omega\sqrt{3\eta +2} }{ \sqrt{3\eta +2} -\sqrt{2-\eta} }\cdot k^2 \notag
\end{align}
By (\ref{eq:3.38}), (\ref{eq:3.42}) and (\ref{eq:3.26}), we obtain that 
\begin{align*}
\lim_{\eta_3 \to\alpha_1} \beta^2\sn^2 \l(\frac{x}{2g} ;k\r) 
&=\lim_{\eta \to 2} \beta^2 \l( \frac{x^2}{4g^2} +O\l( 2-\eta \r) \r)\\
&= \lim_{\eta \to 2} \l[ \frac{ 2\eta^2\omega\sqrt{3\eta +2} }{ \sqrt{3\eta +2} -\sqrt{2-\eta} }\cdot k^2x^2 +O\l( \sqrt{2-\eta}\r) \r] \\
&=4\omega x^2 =(cx)^2. 
\end{align*}
Hence, we deduce that 
\begin{align*}
\lim_{L\to\infty} \psi^L_{c^2/4,c} (x)&= \lim_{\eta_3 \to \alpha_1} \eta_3 \l[\frac{\mathrm{dn}^2 \l( \frac{x}{2g} ; k\r)}{1+ \beta^2 \mathrm{sn}^2 \l( \frac{x}{2g} ; k\r)} \r] \\
&=\frac{4\sqrt{\omega} +2c}{ 1+(cx)^2 } \\
&= \frac{4c}{ 1+(cx)^2 } =\Phi^2_{c^2/4,c} (x) .
\end{align*}
This completes the proof of Theorem \ref{thm:1.1}. 
$\hfill\square$
\section{Long-period limit procedure}
\label{sec:4}
\subsection{$L^2$-convergence} 
First, we discuss the convergence of the mass $\|\Phi_{\omega ,c}^L \|_{L^2(\T_{2L})}^2$ in the long-period limit. We recall that the mass of the soliton on the whole line is given by
\begin{align}
\label{eq:4.1}
\| \Phi_{\omega ,c} \|_{L^2(\R )}^2= 8\tan^{-1} \sqrt{ \frac{2\sqrt{\omega}+c}{2\sqrt{\omega}-c} }.
\end{align}
Our main purpose in this subsection is to prove the following theorem.
\begin{theorem}
\label{thm:4.1}
Let $(\omega ,c)$ satisfy \textup{(\ref{WC})}. Then, we have
\begin{align}
\lim_{L\to \infty}\| \Phi^L_{\omega ,c} \|_{L^2(\T_{2L})}^2 = \| \Phi_{\omega ,c} \|_{L^2(\R )}^2 .
\end{align}
\end{theorem}
\begin{proof}
We calculate the mass of traveling waves on $\T_{2L}$ as 
\begin{align*}
\| \Phi^L_{\omega ,c} \|_{L^2(\T_{2L})}^2 &=2\int_{0}^L \eta_3 \frac{\dn^2 \l( \frac{x}{2g} ; k\r)}{1+ \beta^2 \sn^2 \l( \frac{x}{2g} ; k\r)} dx \\
&=4g\eta_3 \int_{0}^{K(k)} \frac{\dn^2 (x;k)}{1+\beta^2 \sn^2 (x;k)} dx,
\end{align*}
where we used $L=2gK(k)$ in the last equality. Applying formula 410.04 in \cite{BF71}, we have
\begin{align}
\label{eq:4.3}
\| \Phi^L_{\omega ,c} \|_{L^2(\T_{2L})}^2 =4g\eta_3 \sqrt{\frac{k^2+\beta^2}{(1+\beta^2 )\beta^2}} G(\mu ,k),
\end{align}
where 
\begin{align}
\label{eq:4.4}
G(\mu , k) &:=K(k)E(\mu ,k')-K(k)F(\mu ,k') +E(k) F(\mu ,k' ), \\
\label{eq:4.5}
\mu &:= \sin^{-1} \sqrt{\frac{\beta^2}{\beta^2 +k^2}}.
\end{align}
We note that $\mu$ is regarded as a function of $\eta_3$ and that $0 < \mu < \frac{\pi}{2}$ when $\alpha_0 <\eta_3 <\alpha_1$. We set 
\begin{align}
\label{eq:4.6}
\mu_1:=\lim_{\eta_3 \to \alpha_1}\mu =\lim_{\eta_3 \to \alpha_1} \sin^{-1} \sqrt{\frac{\beta^2}{\beta^2 +k^2}}.
\end{align}
Case 1: $\omega >\frac{c^2}{4} $. 
By Lemma \ref{lem:3.4} and Lemma \ref{lem:3.6}, we have
\begin{align}
\label{eq:4.7}
\lim_{\eta_3 \to \alpha_1} 4g\eta_3 \sqrt{\frac{k^2+\beta^2}{(1+\beta^2 )\beta^2}} 
&= \lim_{\eta_3 \to \alpha_1} \frac{4g\eta_3}{\beta} \\
&= \frac{8(2\sqrt{\omega}+c)}{\sqrt{(2\sqrt{\omega}+c)(2\sqrt{\omega} -c)} } \sqrt{\frac{2\sqrt{\omega}-c}{2\sqrt{\omega}+c}} \notag \\
&=8 .\notag 
\end{align}
By the Taylor expansion, we have 
\begin{align}
\label{eq:4.8}
\frac{1}{\sqrt{1-x}} &=1+\sum_{n=1}^{\infty} \frac{(2n-1)!!}{(2n)!!} x^n, \\
\label{eq:4.9}
\sqrt{1-x} &= 1 - \sum_{n=1}^{\infty} \frac{1}{2n-1}\cdot \frac{(2n-1)!!}{(2n)!!} x^n
\end{align}
for all $|x|<1$. Let $\tau \in (0, \frac{\pi}{2})$. Applying (\ref{eq:4.8}) and (\ref{eq:4.9}), we have
\begin{align}
\label{eq:4.10}
E(\tau ,k') &= \int_{0}^{\tau} \sqrt{1-k'^2 \sin^2 \theta } d\theta \\
&=\theta - \sum_{n=1}^{\infty} \frac{1}{2n-1}\cdot \frac{(2n-1)!!}{(2n)!!} k'^{2n} \int_{0}^{\tau} \sin^{2n} \theta d\theta ,\notag \\
\label{eq:4.11}
F(\tau ,k') &= \int_{0}^{\tau} \frac{d\theta }{\sqrt{1-k'^2 \sin^2 \theta }} \\
&=\theta + \sum_{n=1}^{\infty} \frac{(2n-1)!!}{(2n)!!} k'^{2n} \int_{0}^{\tau} \sin^{2n} \theta d\theta .\notag
\end{align}
By (\ref{eq:4.10}) and (\ref{eq:4.11}), we have
\begin{align}
\label{eq:4.12}
\sup_{0\leq \tau \leq \frac{\pi}{2}} | E(\tau ,k') -F(\tau , k')| &\leq \frac{\pi}{2}\sum_{n=1}^{\infty} \frac{2n}{2n-1}\cdot \frac{(2n-1)!!}{(2n)!!} k'^{2n} \\
&\leq Ck'^2 ,\notag
\end{align}
where $C$ is independent of $k'$. By (\ref{eq:2.2}) and (\ref{eq:4.12}), we deduce that
\begin{align}
\label{eq:4.13}
\sup_{0 \leq \tau \leq \frac{\pi}{2}} |K(k) (E(\tau , k') -F(\tau , k')) | 
\leq C k'^2 \l( -\log \frac{k'}{4}\r) \to 0 
\end{align}
as $k' \to 0$. Especially, we deduce that
\begin{align}
\label{eq:4.14}
\lim_{\eta_3 \to \alpha_1}K(k) (E(\mu , k') -F(\mu , k')) =0 .
\end{align}
By (\ref{eq:4.6}) and Lemma \ref{lem:3.6}, we have
\begin{align*}
\sin \mu_1 = \lim_{\eta_3 \to \alpha_1} \sqrt{\frac{\beta^2}{\beta^2 +k^2}} 
=\sqrt{\frac{2\sqrt{\omega}+c}{4\sqrt{\omega}} }.
\end{align*}
Since
\begin{align*}
\sin^2 \mu_1 =\frac{2\sqrt{\omega}+c}{4\sqrt{\omega}},~\cos^2 \mu_1 =
\frac{2\sqrt{\omega}-c}{4\sqrt{\omega}} 
\end{align*}
and $\mu_1 \in [0, \frac{\pi}{2}]$, we deduce that
\begin{align}
\label{eq:4.15}
\mu_1 =\tan^{-1} \sqrt{ \frac{2\sqrt{\omega}+c}{2\sqrt{\omega}-c} }.
\end{align}
By (\ref{eq:4.3}), (\ref{eq:4.7}), (\ref{eq:4.14}) and (\ref{eq:4.15}), we obtain that
\begin{align}
\lim_{L\to \infty}\| \Phi^L_{\omega ,c} \|_{L^2(\T_{2L})}^2 &=\lim_{\eta_3 \to \alpha_1}4g\eta_3 \sqrt{\frac{k^2+\beta^2}{(1+\beta^2 )\beta^2}} G(\mu ,k)  \\
&=8E(1) F(\mu_1 ,0) \notag \\
&=8\mu_1 \notag \\
&=8\tan^{-1} \sqrt{ \frac{2\sqrt{\omega}+c}{2\sqrt{\omega}-c} }
=  \| \Phi_{\omega ,c} \|_{L^2(\R )}^2 . \notag
\end{align}
Case 2: $\omega =c^2/4 $ and $c>0$.  Since
\begin{align}
\label{eq:4.17}
\begin{array}{lll}
k^2 \to \frac{1}{2}, &k'^2 \to \frac{1}{2}, &\beta \to \infty ,\\[3pt]
\eta_1 \to 0, & \eta_2 \to 0
\end{array}
\end{align}
as $\eta_3 \to \alpha_1$ in this case, we need to modify the previous calculation. By (\ref{eq:4.17}), we have
\begin{align}
\label{eq:4.18}
\frac{k^2+\beta^2}{1+\beta^2} \to 1
\end{align}
as $\eta_3 \to \alpha_1$. By using the definition of $k, g$ and $\beta$, we have
\begin{align}
\label{eq:4.19}
\frac{4g\eta_3}{\beta} &=\frac{8\eta_3}{\sqrt{\eta_3 ( \eta_2 -\eta_1 )}} \cdot \sqrt{\frac{-\eta_1}{\eta_3}} \cdot \sqrt{\frac{\eta_3 (\eta_2 -\eta_1 )}{-\eta_{1}(\eta_3-\eta_2 )}} \\
&= 8\sqrt{ \frac{\eta_3}{\eta_3-\eta_2} }. \notag
\end{align}
By (\ref{eq:4.18}) and (\ref{eq:4.19}), we obtain that
\begin{align}
\label{eq:4.20}
\lim_{\eta_3 \to \alpha_1} 4g\eta_3 \sqrt{\frac{k^2+\beta^2}{(1+\beta^2 )\beta^2}} 
&= \lim_{\eta_3 \to \alpha_1} \frac{4g\eta_3}{\beta} \\
&=\lim_{\eta_3 \to \alpha_1} 8\sqrt{ \frac{\eta_3}{\eta_3-\eta_2} } \notag \\
&=8.\notag
\end{align}
By (\ref{eq:4.17}), we note that
\begin{align}
\mu_1 =\lim_{\eta_3 \to \alpha_1} \sin^{-1} \sqrt{\frac{\beta^2}{\beta^2 +k^2}}
= \sin^{-1} 1=\frac{\pi}{2} .
\end{align}
Hence, we obtain that
\begin{align}
\label{eq:4.22}
\lim_{\eta_3 \to \alpha_1} G(\mu ,k) &= G(\mu_1, {\scriptstyle \frac{1}{\sqrt{2}}}) \\
&=(KE'-KK'+EK') ({\scriptstyle \frac{1}{\sqrt{2}}}) \notag \\
&=\frac{\pi}{2} , \notag
\end{align}
where we used the Legendre relation (\ref{eq:2.4}) in the last equality. By (\ref{eq:4.3}), (\ref{eq:4.20}) and (\ref{eq:4.22}), we obtain that
\begin{align}
\lim_{L\to \infty}\| \Phi^L_{c^2/4 ,c} \|_{L^2(\T_{2L})}^2 &=\lim_{\eta_3 \to \alpha_1}4g\eta_3 \sqrt{\frac{k^2+\beta^2}{(1+\beta^2 )\beta^2}} G(\mu ,k)  \\
&=8G(\mu_1, {\scriptstyle \frac{1}{\sqrt{2}}}) \notag \\
&=4\pi
=  \| \Phi_{c^2/4 ,c} \|_{L^2(\R )}^2 . \notag
\end{align}
This completes the proof.
\end{proof}
Next, we prove the following theorem. This is the partial statement of Theorem~\ref{thm:1.2}.
\begin{theorem}
\label{thm:4.2}
Let $(\omega ,c)$ satisfy \textup{(\ref{WC})}. Then, we have
\begin{align}
\label{eq:4.24}
\lim_{L\to \infty} \| \Phi_{\omega ,c}^L -\Phi_{\omega ,c}\|_{H^m ([-L,L])} =0
\end{align}
for all $m=0 ,1, 2$.
\end{theorem} 
To prove Theorem~\ref{thm:4.2}, we recall the Br\'ezis--Lieb lemma.
\begin{lemma}[{\cite{BL83}}]
\label{lem:4.3}
Let $1\leq p < \infty$. Let $\{f_L\}$ be a bounded sequence in $L^p(\R)$ and $f_L \to f$ a.e. in $\R$ as $L\to \infty$. Then we have 
\begin{align*}
\| f_L\|_{L^p}^p - \| f_L-f\|_{L^p}^p - \|f \|_{L^p}^p \to 0
\end{align*}
as $L \to \infty$.
\end{lemma}
\begin{proof}[Proof of Theorem~\ref{thm:4.2}]
We consider $\Phi_{\omega ,c}^L$ as the function defined on $\R$. More precisely, we extend the function $\Phi_{\omega ,c}^L$ as
\begin{align}
\label{eq:4.25}
\Phi_{\omega ,c}^L (x)=
\l\{
\begin{array}{ll}
\ds \Phi_{\omega ,c}^L (x) &\ds \text{if}~x \in [-L ,L],\\[3pt]
\ds \Phi_{\omega ,c}^L (x-2Lk) &\ds \text{if}~x \in [(2k-1)L ,(2k+1)L],~k\in \Z\setminus \{ 0\} .
\end{array}
\r. 
\end{align}
We set $f_L=\chi_{[-L,L]} \Phi_{\omega ,c}^L$ and $f=\Phi_{\omega ,c}$. By Theorem~\ref{thm:1.1} and Theorem~\ref{thm:4.1}, we have
\begin{align*}
&f_L (x) \to f(x)~\text{for all}~ x \in \R, \\
&\| f_L\|_{L^2(\R )}^2 \to \| f\|_{L^2 (\R )}^2
\end{align*}
as $L \to \infty$. Applying Lemma~\ref{lem:4.3}, we obtain
\begin{align}
\label{eq:4.26}
\lim_{L\to \infty}\| f_L -f\|_{L^2(\R )}^2 = 0.
\end{align}
Since $f \in L^2(\R)$, we have
\begin{align}
\label{eq:4.27}
\lim_{L\to \infty} \| f\|_{L^2(|x|\geq L)}^2 =\lim_{L\to \infty} \| \Phi_{\omega ,c}\|_{L^2(|x|\geq L)}^2 =0.
\end{align}
By (\ref{eq:4.26}) and (\ref{eq:4.27}), we deduce that
\begin{align}
\label{eq:4.28}
\lim_{L\to \infty} \| \Phi_{\omega ,c}^L -\Phi_{\omega ,c}\|_{L^2 ([-L,L])}=
\lim_{L\to \infty} \| f_L-f \|_{L^2(|x|\leq L)} =0.
\end{align}

Next we prove 
\begin{align}
\label{eq:4.29}
\lim_{L\to \infty} \| \del^2 \Phi_{\omega ,c}^L- \del^2 \Phi_{\omega ,c} \|_{L^2([-L,L])} =0.
\end{align}
We note that $\Phi_{\omega ,c}^L$ and $\Phi_{\omega ,c}$ satisfy the same equation (\ref{ELL}). 
For each $L>0$, we have 
\begin{align}
\label{eq:4.30}
|\Phi_{\omega ,c}^L(x)|^2 \leq \eta_3  ~\text{for all}~ x \in [-L,L],
\end{align}
since $\sqrt{\eta_3}$ is maximum value of $\Phi_{\omega ,c}^L$. By the exciplit formula (\ref{eq:1.7}) of the soliton, we have
\begin{align}
\label{eq:4.31}
\| f\|_{L^{\infty}(\R)}^2 =\Phi_{\omega ,c}^2 (0) =4\sqrt{\omega}+2c=\alpha_1 .
\end{align}
By (\ref{eq:4.30}), (\ref{eq:4.31}) and (\ref{eq:4.28}), we deduce that
\begin{align*}
\| f_L^3-f^3\|_{L^2([-L,L])} &\leq C (\| f_L\|_{L^{\infty}([-L,L])}^2 +\| f\|_{L^{\infty}([-L,L])}^2 ) \| f_L-f\|_{L^2([-L,L])} \\
&\leq C(\eta_3 +\alpha_1 )\| f_L-f\|_{L^2([-L,L])}  \\
&\leq 2C\alpha_1 \| f_L-f\|_{L^2([-L,L])} \underset{L\to\infty}{\longrightarrow} 0.
\end{align*}
Similarly, we have
\begin{align*}
\| f_L^5-f^5\|_{L^2([-L,L])} &\leq C (\| f_L\|_{L^{\infty}([-L,L])}^4 +\| f\|_{L^{\infty}([-L,L])}^4 ) \| f_L-f\|_{L^2([-L,L])} \\
&\leq 2C\alpha_1^2\| f_L-f\|_{L^2([-L,L])} \underset{L\to\infty}{\longrightarrow} 0.
\end{align*}
Hence, by using the equation (\ref{ELL}), we deduce that
\begin{align*}
\| \del^2 \Phi_{\omega ,c}^L- \del^2 \Phi_{\omega ,c} \|_{L^2([-L,L])} &\leq \l(\omega- \frac{c^2}{4}\r) \| f_L-f\|_{L^2([-L,L])}
\\
&\quad +\frac{c}{2} \| f_L^3-f^3\|_{L^2([-L,L])}+\frac{3}{16}\| f_L^5-f^5\|_{L^2([-L,L])}\\
&~ \underset{L\to\infty}{\longrightarrow} 0.
\end{align*}
Finally, by integration by parts, we obtain that
\begin{align*}
\| \del \Phi_{\omega ,c}^L- \del \Phi_{\omega ,c} \|_{L^2([-L,L])}^2
&=-\int_{-L}^L \l( \del^2 \Phi_{\omega ,c}^L- \del^2 \Phi_{\omega ,c} \r) \l( \Phi_{\omega ,c}^L- \Phi_{\omega ,c}\r) dx\\
&\leq \| \del^2 \Phi_{\omega ,c}^L- \del^2 \Phi_{\omega ,c} \|_{L^2([-L,L])}\| \Phi_{\omega ,c}^L-  \Phi_{\omega ,c} \|_{L^2([-L,L])}\\
&\underset{L\to\infty}{\longrightarrow} 0.
\end{align*}
This completes the proof.
\end{proof}
To prove the estimate (\ref{eq:4.24}) for $m \geq 3$ by using the equation (\ref{ELL}), we need to control $L^{\infty}$-norm of lower derivative $\del_x^k \Phi_{\omega ,c}^L$ where $k=1, 2, \cdots, m-1$. To achieve this, we discuss $L^{\infty}$-convergence of $\Phi_{\omega ,c}^L$ in next subsection.
\subsection{$L^{\infty}$-convergence}
In this subsection, we mainly prove the following proposition.
\begin{proposition}
\label{prop:4.4}
Let $(\omega ,c)$ satisfy \textup{(\ref{WC})}. Then, we have
\begin{align}
\label{eq:4.32}
\lim_{L\to \infty} \| \Phi_{\omega ,c}^L -\Phi_{\omega ,c}\|_{L^{\infty} ([-L,L])} =0.
\end{align}
\end{proposition}
\begin{proof}
Since $\Phi_{\omega ,c}^L$ and $\Phi_{\omega ,c}$ are even functions, it is enough to consider the interval $[0,L]$. We use the same notation in the proof of Theorem~\ref{thm:4.2}. By fundamental theorem of calculus, we have
\begin{align*}
f_L^2(x) &=f_L(0)^2+\int_{0}^x \frac{d}{dy} f_L^2(y) dy, \\
f^2(x) &=f(0)^2+\int_{0}^x \frac{d}{dy} f^2(y) dy,
\end{align*}
for all $x \in [0,L]$. Since $f_L(0)^2=\eta_3$ and $f(0)^2=\alpha_1$, we have
\begin{align}
\label{eq:4.33}
f_L^2(x) -f(x)^2 =\eta_3 -\alpha_1 +2\int_{0}^{x} (f_Lf_L' -ff') dy
\end{align}
for all $x \in [0,L]$. By Theorem~\ref{thm:4.1}, we note that
\begin{align}
\label{eq:4.34}
\sup_{L_0<L<\infty} \| f_L\|_{L^2([0,L])} \leq C=C(\| f\|_{L^2(\R )}).
\end{align}
Applying H\"older's inequality and (\ref{eq:4.34}), we deduce that
\begin{align*}
\int_{0}^L |f_Lf_L'-ff'|dy &\leq \| f_L\|_{L^2([0,L])}\| f_L'-f'\|_{L^2([0,L])}+\| f'\|_{L^2([0,L])}\| f_L -f\|_{L^2([0,L])}\\
&\leq C\| f_L -f\|_{H^1([0,L])} .
\end{align*}
Combined with (\ref{eq:4.33}), it follows from Theorem~\ref{thm:4.2} that
\begin{align}
\label{eq:4.35}
\| f_L^2-f^2\|_{L^{\infty}([0,L])} &\leq |\eta_3  -\alpha_1 |+2\int_{0}^L |f_Lf_L'-ff'|dy\\
&\leq |\eta_3 -\alpha_1 |+C\| f_L -f\|_{H^1([0,L])}  
\underset{L\to\infty}{\longrightarrow} 0. \notag
\end{align}
By using the elementary inequality
\begin{align*}
|\sqrt{x}-\sqrt{y}| \leq \sqrt{|x-y|}~\text{for all}~ x,y \geq 0
\end{align*}
and (\ref{eq:4.35}), we deduce that
\begin{align}
\| f_L-f\|_{L^{\infty}([0,L])} &\leq \sqrt{\| f_L^2-f^2\|_{L^{\infty}([0,L])} }
\underset{L\to\infty}{\longrightarrow} 0. \notag
\end{align}
This completes the proof.
\end{proof}
\begin{remark}
We can also prove Proposition~\ref{prop:4.4} directly without using the result of Theorem~\ref{thm:4.2}. Given a $\eps >0$. By the decay of $\Phi_{\omega, c}$ and the pointwise convergence in Theorem~\ref{thm:1.1}, there exists $L_0>0$ such that
\begin{align}
\label{eq:4.36}
|\Phi_{\omega, c} (L_0)| <\eps ,~|\Phi_{\omega, c}^L (L_0) | < 2\eps 
\end{align}
for large $L>L_0>0$. Both $\Phi_{\omega, c}^L$ and $\Phi_{\omega, c}$ are radial and decreasing functions, we deduce that
\begin{align}
\label{eq:4.37}
 \| \Phi_{\omega, c}\|_{L^{\infty}(L_0\leq |x|\leq L)} <\eps ,~\| \Phi_{\omega, c}^L \|_{L^{\infty}(L_0\leq |x|\leq L )}<2\eps 
\end{align}
for large $L>0$. On the other hand, by reviewing the proof of Theorem~\ref{thm:1.1}, it is easily verified that
\begin{align}
\label{eq:4.38}
\lim_{L\to \infty} \| \Phi_{\omega, c}^L-\Phi_{\omega, c} \|_{L^{\infty}([-L_0 ,L_0])}=0 .
\end{align}
By (\ref{eq:4.37}) and (\ref{eq:4.38}), we obtain
\begin{align*}
\limsup_{L\to \infty}\|\Phi_{\omega, c}^L-\Phi_{\omega, c} \|_{L^{\infty}([-L,L])}\leq 2\eps .
\end{align*}
This gives an alternative proof of Proposition~\ref{prop:4.4}.
\end{remark}

The following proposition follows from Proposition~\ref{prop:4.4} and similar discussion on the proof of Theorem~\ref{thm:4.2}.
\begin{proposition}
\label{prop:4.5}
Let $(\omega ,c)$ satisfy \textup{(\ref{WC})}. Then, we have
\begin{align}
\label{eq:4.39}
\lim_{L\to \infty} \| \Phi_{\omega ,c}^L -\Phi_{\omega ,c}\|_{C^m([-L,L])} =0
\end{align}
for all $m=0 ,1, 2$.
\end{proposition}
\subsection{Proof of Theorem \ref{thm:1.2}~ and Theorem \ref{thm:1.3}}
\begin{proof}[Proof of Theorem~\ref{thm:1.3}]
It is proved by differentiating the equation (\ref{ELL}) and by applying Proposition~\ref{prop:4.5} and the induction. We omit the detail.
\end{proof}
\begin{proof}[Proof of Theorem~\ref{thm:1.2}]
It is proved similarly as Theorem~\ref{thm:1.3} by using the induction. We prove only case $m=3$. By differentiating the equation (\ref{ELL}), we have
\begin{align}
\label{eq:4.40}
-\Phi ''' +\l(\omega- \frac{c^2}{4}\r) \Phi '+\frac{3}{2}c\Phi^2 \Phi ' - \frac{15}{16} \Phi^4 \Phi ' =0,
\end{align}
By Proposition~\ref{prop:4.5}, we note that
\begin{align}
\label{eq:4.41}
\sup_{L_0<L<\infty} \| f_L\|_{W^{1,\infty}([-L,L])} \leq C=C(\| f\|_{W^{1,\infty}(\R )}).
\end{align}
By (\ref{eq:4.41}) and Theorem~\ref{thm:4.2}, we have
\begin{align*}
\| f_L^2f_L'-f^2 f'\|_{L^2([-L,L])} &\leq \| f_L'\|_{L^{\infty}([-L,L])} \| f_L^2-f^2\|_{L^2([-L,L])} \\
&\quad +\| f^2\|_{L^{\infty}([-L,L])}\| f_L' -f'\|_{L^2([-L,L])} \\
 &\leq C\| f_L-f\|_{H^1([-L,L])} \underset{L\to\infty}{\longrightarrow} 0.
\end{align*}
Similarly, we have
\begin{align*}
\lim_{L\to \infty}\| f_L^4f_L'-f^4 f'\|_{L^2([-L,L])} =0 .
\end{align*}
By using the equation (\ref{eq:4.40}), we deduce that
\begin{align*}
\lim_{L\to \infty} \| \del^3 \Phi_{\omega ,c}^L- \del^3 \Phi_{\omega ,c} \|_{L^2([-L,L])}^2 =0.
\end{align*}
This completes the proof.
\end{proof}
\appendix
\section{Monotonicity of the modulus} 
\label{sec:A}
Here, we prove Proposition \ref{prop:3.5}. We recall that $k^2$ is a function of $\eta$ as 
\begin{align*}
k^2 &=\frac{3\eta^2 +(\sqrt{f_s(\eta )}-6s)\eta +2(s^2 -1)}{2\eta\sqrt{f_s(\eta )}},
\end{align*}
where $f_s(\eta )$ is defined by
\begin{align*}
f_s(\eta ) :=-3\eta^2 +4s\eta +4
\end{align*}
for $-1<s\leq 1$ and $\beta_0 <\eta <\beta_1$. We also recall that $\beta_0$ and $\beta_1$ are defined by
\begin{align*}
\beta_0 &=\beta_0 (s)=\frac{1}{3}\l( 2s+\sqrt{3+s^2}\r) ,\\
\beta_1 &=\beta_1 (s)=1+s.
\end{align*}
We define the function $b$ by
\begin{align}
\label{eq:A.1}
b=b_s(\eta ):=\eta^2 f_s(\eta )=-3\eta^4 +4s\eta^3 +4\eta^2 .
\end{align}
Note that by Lemma \ref{lem:3.2} $b=b_s(\eta )$ is positive for $-1<s\leq 1$ and $\beta_0 <\eta <\beta_1$. We differentiate $k^2$ with respect to $\eta$ as
\begin{align*}
\frac{dk^2}{d\eta} &=\frac{1}{2b}\l[ \l( 6\eta -6s +\frac{d\sqrt{b}}{d\eta}\r)\sqrt{b}-
\l( 3\eta^2 +\sqrt{b} -6s\eta +2(s^2-1 )\r) \frac{d\sqrt{b}}{d\eta}\r] \\
&=\frac{1}{2b}\l[ 6( \eta -s )\sqrt{b}-
\l( 3\eta^2 -6s\eta +2(s^2-1 )\r) \frac{d\sqrt{b}}{d\eta}\r] .
\end{align*}
A direct computation shows that
\begin{align}
\label{eq:A.2}
\frac{dk^2}{d\eta} 
&=\frac{\eta}{b\sqrt{b}}\bigl( 6s\eta g_s(\eta )+4(1-s^2) \bigr) ,
\end{align}
where the function $g_s$ is defined by
\begin{align*}
g_s(\eta )&:=-\l( \eta -(s-1)\r) \l( \eta -(s+1)\r) .
\end{align*}
We note that a positive zero of $g_s(\eta )$ is given by $\beta_1=s+1$. Since $g_s(0)=1-s^2 \geq 0$ for $-1<s\leq 1$, we have $g_s(\eta) >0$ for $\beta_0 <\eta <\beta_1$.
Therefore, if $0< s\leq 1$, by (\ref{eq:A.2}) we obtain
\begin{align*}
\frac{dk^2}{d\eta}
&\geq \frac{6s\eta g_s(\eta )}{b\sqrt{b}}>0.
\end{align*}
If $s=0$, by (\ref{eq:A.2}) we obtain
\begin{align*}
\frac{dk^2}{d\eta}=\frac{4\eta}{b\sqrt{b}}>0.
\end{align*}
Finally, we consider the case $-1<s<0$. We note that $\beta_0$ gives a positive maximal point of $\eta\mapsto \eta g_s(\eta)$. Therefore, if $-1<s<0$, we have
\begin{align}
\label{eq:A.3}
6s\eta g_s(\eta )+4(1-s^2)
&>6s\beta_0 g_s(\beta_0 )+4(1-s^2)=:h(s).
\end{align}
From the definitions of $\beta_0$ and $g_s$, $h(s)$ is rewritten as
\begin{align}
\label{eq:A.4}
h(s)&=\frac{4}{9}\l( -s^4+(3+s^2)^{3/2}s+9\r) .
\end{align}
We note that $h(0)=4$, $h(-1)=0$, and $s\mapsto h(s)$ is strictly increasing on the interval $[-1,0]$. Hence, from (\ref{eq:A.2}) and (\ref{eq:A.3}), we deduce that
\begin{align*}
\frac{dk^2}{d\eta}&=\frac{\eta}{b\sqrt{b}}\bigl( 6s\eta g_s(\eta )+4(1-s^2) \bigr)  \\
&>\frac{\eta}{b\sqrt{b}}\cdot h(s)>\frac{\eta}{b\sqrt{b}}\cdot h(-1)=0.
\end{align*}
This completes the proof of Proposition \ref{prop:3.5}.
\section{Monotonicity of the fundamental period}
\label{sec:B}
Here, we prove Proposition \ref{prop:3.7}. We recall that the fundamental period $T_{\psi}$ is defined by
\begin{align*}
T_{\psi}=\frac{8}{\sqrt{\eta_3 (\eta_2 -\eta_1 )}} K(k) .
\end{align*}
We note that 
\begin{align*}
\eta_3 (\eta_2 -\eta_1 ) =16\omega\eta \sqrt{f_s(\eta )}=16\omega \sqrt{b_s(\eta )},
\end{align*}
where $b_s(\eta )$ is defined by (\ref{eq:A.1}). Hence, $T_{\psi}$ is rewritten as
\begin{align}
\label{eq:B.1}
T_{\psi}  =\frac{2}{\sqrt{\omega}b_s(\eta )^{1/4}}K(k (\eta )),
\end{align}
where $-1<s\leq 1$ and $\beta_0 <\eta <\beta_1$. We differentiate $T_{\psi}$ with respect to $\eta$ as 
\begin{align}
\label{eq:B.2}
\frac{ \sqrt{\omega} }{2}\cdot\frac{dT_{\psi}}{d\eta}&=\frac{dK}{dk}\frac{dk}{d\eta}\cdot\frac{1}{b_s^{1/4}}+
K\cdot\l( -\frac{1}{4b_s^{5/4}}\r) \frac{db_s}{d\eta} \\
&=\frac{1}{b_s^{5/4}}\l( \frac{dK}{dk}\frac{dk}{d\eta}\cdot b_s -K\eta a_s(\eta )\r) ,
\notag
\end{align}
where the function $a_s$ is defined by
\begin{align}
\label{eq:B.3}
a_s(\eta ) :=-3\eta^2 +3s\eta +2.
\end{align}
We note that 
\begin{align}
\label{eq:B.4}
\gamma =\gamma (s) :=\frac{3s+\sqrt{9s^2+24}}{6}
\end{align}
gives a positive zero of $a_s(\eta )$. Since $a_s(0)=2$, we have $a_s (\eta) >0$ for $0<\eta <\gamma$. 
When $-1<s\leq 1$, we have
\begin{align*}
\gamma \leq \beta_1 &\iff  \frac{3s+\sqrt{9s^2+24}}{6}\leq 1+s \\
&\iff -\frac{1}{3}\leq s.
\end{align*}
On the other hand, we have
\begin{align*}
\beta_0 <\gamma &\iff \frac{1}{3}\l( 2s+\sqrt{s^2+3}\r) <\frac{1}{6}\l( 3s+\sqrt{9s^2+24}\r)\\
&\iff s<\sqrt{s^2+3} .
\end{align*}
Since the last inequality holds for any $s\in\R$, we have $\beta_0 <\gamma$ for $-1<s\leq 1$. Hence, the following three cases can be considered.
\begin{enumerate}[(a)]
\setlength{\itemsep}{3pt}
\item $\ds -\frac{1}{3} <s\leq 1$, $\gamma\leq \eta <\beta_1$.

\item $\ds -\frac{1}{3} <s\leq 1$, $\beta_0 <\eta <\gamma$.

\item  $\ds -1<s \leq -\frac{1}{3}$, $\beta_0 <\eta <\beta_1(\leq \gamma)$.
\end{enumerate}
Since $\ds \frac{dK}{dk}>0$, and $\ds \frac{dk}{d\eta} >0$ from Appendix \ref{sec:A}, we note that the first term on the RHS of (\ref{eq:B.2}) is positive. In the case (a), since $a(\eta) \leq 0$, by (\ref{eq:B.2}) we deduce that
\begin{align}
\label{eq:B.5}
\frac{dT_{\psi}}{d\eta} >0 .
\end{align}
In the latter two cases, since $a_s (\eta ) >0$, we need to calculate a little more carefully. But, by using the formula
\begin{align*}
\frac{dK}{dk}=\frac{1}{kk'^2}(E-k'^2K)
\end{align*}
and (\ref{eq:A.2}), one can prove that (\ref{eq:B.5}) holds in these cases. We omit the detail and refer to \cite{A07, AN09} as similar arguments.
\section*{Acknowledgments}
The author would like to thank Tohru Ozawa for his advice and encouragement. He is also grateful to Masahito Ohta for many useful discussions on topics closely related to this work, especially on stability of solitons. This work was completed during the author's stay at Instituto de Matem\'{a}tica Pura e Aplicada (IMPA). The author is grateful to Felipe Linares for his hospitality. 
The author also thanks Noriyoshi Fukaya and Takahisa Inui for helpful discussions, and the anonymous referees for their helpful comments on the references. This work was supported by Grant-in-Aid for JSPS Fellows 17J05828 and Top Global University Project, Waseda University.


\end{document}